\documentclass[reqno,11pt]{amsart}
\usepackage{a4wide,color,eucal,enumerate,mathrsfs}
\usepackage[normalem]{ulem}
\usepackage{amsmath,amssymb,epsfig,amsthm} 
\usepackage[latin1]{inputenc}
\usepackage{psfrag}

\def\rr{{\mathbb R}}
\def\rn{{{\rr}^n}}

\def\nn{{\mathbb N}}

\def\fz{\infty}
\def\az{\alpha}

\def\dist{{\mathop\mathrm{\,dist\,}}}
\def\loc{{\mathop\mathrm{\,loc\,}}}

\def\lz{\lambda}
\def\dz{\delta}

\def\ez{\epsilon}
\def\kz{\kappa}
\def\bz{\beta}

\def\gz{{\gamma}}

\def\vz{\varphi}

\def\wz{\widetilde}

\def\bint{{\ifinner\rlap{\bf\kern.25em--}
\int\else\rlap{\bf\kern.45em--}\int\fi}\ignorespaces}

\def\bbint{{\ifinner\rlap{\bf\kern.35em--}
\hspace{0.078cm}\int\else\rlap{\bf\kern.45em--}\int\fi}\ignorespaces}

\def\esup{\mathop\mathrm{\,esssup\,}}

\def\osc{ \mathop \mathrm{\, osc\,} }

\newtheorem{thm}{Theorem}[section]
\newtheorem{lem}[thm]{Lemma}
\newtheorem{cor}[thm]{Corollary}
\numberwithin{equation}{section}

\theoremstyle{remark}
\newtheorem{rem}[thm]{Remark}

\def\bint{{\ifinner\rlap{\bf\kern.35em--}
\int\else\rlap{\bf\kern.45em--}\int\fi}\ignorespaces}

\title[Sobolev regularity for  inhomogeneous $\infty$-Laplace equation]
{Some   sharp Sobolev regularity  for  inhomogeneous $\infty$-Laplace equation in plane}
\author{Herbert Koch,  Yi Ru-Ya Zhang and Yuan Zhou}
\address{H. Koch: Institute of Mathematics, Bonn University, Endenicher Allee 60, Bonn 53115, Germany}
\email{koch@math.uni-bonn.de}
\address{Y. Zhang: Hausdorff Center for Mathematics, Endenicher Allee 62, Bonn 53115, Germany}
\email{yizhang@math.uni-bonn.de}
\address{Y. Zhou: Department of Mathematics, Beihang University, Beijing 100191, P.R. China}
\email{yuanzhou@buaa.edu.cn}
\date{\today}

\begin{document}

 \allowdisplaybreaks
\arraycolsep=1pt
\maketitle
\begin{center}
\begin{minipage}{13cm}
{\bf Abstract.}
Suppose  $\Omega\Subset \mathbb R^2$ and  $f\in BV_\loc(\Omega)\cap C^0(\Omega)$ with $|f|>0$ in $\Omega$.
Let $u\in C^0(\Omega)$ be a viscosity solution to the  inhomogeneous $\infty$-Laplace equation
$$
-\Delta_{\infty} u :=-\frac12\sum_{i=1}^2(|Du|^2)_iu_i= -\sum_{i,j=1}^2u_iu_ju_{ij}   =f \quad {\rm in}\ \Omega.
$$
The following are proved in this paper.
\begin{enumerate}
\item[(i)] For  $ \alpha > 3/2$,  we have $|Du|^{\alpha}\in W^{1,2}_\loc(\Omega)$,
which is (asymptotic) sharp when $ \alpha  \to 3/2$.  Indeed, the function $w(x_1,x_2)=-x_1^ {4/3} $
is a viscosity solution to $-\Delta_\fz w=\frac{4^3}{3^4}$ in $\rr^2$. For any $p> 2$,
  $|Dw|^\alpha \notin W^{1,p}_\loc(\rr^2)$ whenever    $\alpha\in(3/2,3-3/p)$.

\item[(ii)] For  $ \alpha \in(0, 3/2]$ and $p\in[1, 3/(3-\alpha))$,  we have
  $|Du|^{\alpha}\in W^{1,p}_\loc(\Omega)$, which is sharp when $p\to 3/(3-\alpha)$. Indeed,
  $ |Dw|^\alpha \notin W^{1,3/(3-\alpha)}_\loc(\rr^2)$.

\item[(iii)] For  $ \ez > 0$,  we have   $|Du|^{-3+\ez}\in L^1_\loc(\Omega )$, which is sharp when $\epsilon\to0$.
Indeed,  $|Dw|^{-3}  \notin L^1_\loc(\rr^2)$.

\item[(iv)] For  $ \alpha > 0$,  we have $$-(|Du|^{\alpha})_iu_i= 2\alpha|Du|^{{ \alpha-2}}f \ \mbox{ almost everywhere in $\Omega$}.$$

\end{enumerate}
Some quantative bounds are also given. 

\end{minipage}
\end{center}
\tableofcontents

\section{Introduction}

Let $n\ge 2$ and $ \Omega $ be a bounded domain (open connected subset) of $\rn$.
 In 1960's, Aronsson \cite{a1,a2,a3,a4}  derived  the $\fz$-Laplace equation
  \begin{equation}\label{e1.d1}
-\Delta_\infty u:=-\frac12 (|Du|^2)_iu_i=-u_iu_ju_{ij}  =0
\quad {\rm{in}}\quad \Omega
\end{equation}   as the Euler-Lagrange equation when   absolutely minimizing     the $L^\infty$-functional
$$F_\fz(u,\Omega)=
  \esup_{  \Omega} \frac12|Du|^2.
$$
Viscosity solutions to \eqref{e1.d1}  as defined by Cradall et al  \cite{CIL} are called $\fz$-harmonic functions;
while by Aronsson \cite{a1,a2,a3,a4},
 an absolute minimizer  is  a local Lipschitz function which  minimizing $F_\fz(v,V)$ in any domain  $V\Subset \Omega $.
In this paper,
  $v_i$ denotes $\frac{\partial v} {\partial x_i}$ if   $v\in C^1(\Omega)$, and
 the distributional  derivation  in direction $x_i$  if   $v\in L^2_\loc(\Omega)$,  and
 $v_{ij}= \frac{\partial^2 v} {\partial x_i\partial x_j}$ if    $v\in C^2(\Omega)$.
Write  $Dv=(v_i)_{i=1}^n$, $D^2v=(v_{ij})_{i,j=1}^n$ and $D^2vDv=(v_{ij}v_j)_{i=1}^n$.
We always use the Einstein summation convention, that is, $v_iw_i=\sum_{i,j=1}^nv_iw_i$ for  vectors $(w_i)_{i=1}^n$
and $(v_i)_{i=1}^n$.

Jensen  \cite{j1993}  identified $\fz$-harmonic functions with absolute minimizers, and moreover,
  established  their existence and uniqueness under Dirichlet boundary.
Their regularity  then
becomes the main issue in this direction.  By \cite{j1993}, they   are always  local Lipschitz, and hence, by Rademacher's Theorem,  are differentiable almost everywhere.
 Crandall-Evans \cite{ce} proved their linear approximation property at each point, which means that
 for each sequence converging to $0$, one can find a subsequence admitting a tangential plane along it.
Moreover, for planar  $\fz$-harmonic functions  $u$,
 via a key observation from planar topology Savin \cite{s05} proved   their interior  $C^1$-regularity;
 later, the interior $C^{1,\,{ \alpha}}$-regularity with $0<\alpha<<1/3$ was established  by
 Evans-Savin \cite{es08}
and the boundary  $C^1$-regularity by Wang-Yu \cite{wy}.
Recently, we \cite{kzz} obtained  the Sobolev $W^{1,2}_\loc$-regularity of $|Du|^{ \alpha}$ for ${ \alpha}>0$,
which is sharp when $ { \alpha}\to0$; moreover, we proved that the distributional determinant $-\det D^2u$ is a nonnegative Radon measure.
For $n$-dimensional $\fz$-harmonic functions with $n\ge3$,
Evans-Smart \cite{es11a,es11b} obtained their  everywhere differentiability via
an approximation approach by exponential harmonic functions.

On the other hand,
 Lu-Wang \cite{lw} considered the   inhomogeneous $\fz$-Laplace equation
\begin{equation}\label{inhom infty equ}
-\Delta_{\infty} u := -\frac12(|Du|^2)_iu_i=-u_iu_ju_{ij}   =f \quad {\rm in}\ \Omega,
\end{equation}
where   $f\in C^0(\Omega)$. Viscosity solutions to  \eqref{inhom infty equ}  are defined as in \cite{CIL}.
Assuming  that $f$ is bounded and  $|f|>0$,
   for any $g\in C^0(\partial \Omega)$  Lu-Wang   proved
   the existence and  uniqueness of  viscosity solutions $u\in C^0(\overline\Omega)$ to  \eqref{inhom infty equ} so that $u=g$ on $\partial \Omega$.
   We summarize in Section 2 the existence,  uniqueness and also maximum principle  used in current paper. 
   But when $f$ changes sign, a counter-example was constructed in \cite{lw}  to show that the uniqueness may fail.
Under $f \ge0$ or $f\le0$, the uniqueness   is  still open.
Similar results for  inhomogeneous normalized $\fz$-Laplace equation were established
in  \cite{lw2,pssw,as12} via different approaches.

The regularity of viscosity solutions  to \eqref{inhom infty equ} is
far from understood.
If  $ f\in C^0(\Omega)$, viscosity solutions to \eqref{inhom infty equ} are known to be local  Lipschitz; see \cite{lw,l} and see also Lemma \ref{lipu} below for a quantative estimate. Lindgren \cite{l}  obtained their linear approximation property.
Assuming additionally $f\in  C^{0,1}(U)$,   everywhere differentiability   was established by
Lindgren \cite{l} (see also  \cite{lmz}) via the approach of Evans-Smart \cite{es11b}.

  The main purpose of this paper is to prove the following Sobolev regularity for inhomogeneous  $\fz$-Laplace equations \eqref{inhom infty equ} in any domain  $ \Omega\Subset\mathbb R ^2$.
We say that $f\in BV_\loc(\Omega)$ if  for any $U\Subset\Omega$, we have
$$
\|f\|_{BV(U)}:=\sup\left \{\int_U f\Phi^i_i\,dx:\ \Phi=(\Phi^1,\Phi^2)\in C_c^1(U;\mathbb R^2),\ \|\Phi\|_{L^\fz(U)}\le1\right\}<\fz.$$
\begin{thm} \label{mainthm}
Suppose $\Omega\Subset\mathbb R^2$   and  $f\in BV_\loc (\Omega)\cap C^0(\Omega)$ with  $|f|>0$ in $\Omega$.
Let  $u\in  C^0(\Omega)$ be any viscosity solution to \eqref{inhom infty equ}.
\begin{enumerate}
\item[(i)]  For   $\alpha>3/2$, we have  $|Du|^{\alpha}\in W^{1,2}_\loc(\Omega)$  and,  $\forall B:=B(x,R)\subset 2B\Subset\Omega$,
\begin{align}\label{ex01}
  \int_B |D|Du|^{\alpha}|^2\,dx
& \le C(\alpha)\frac1{R^2}\int_{2B}|Du|^{2\alpha}\,dx\nonumber\\
&\quad+ C(\alpha)\| f\| _{BV(2B)}\left [ \frac1R\| u\|_{C^0(2\overline B)}+(R\|f\|_{C^0(2\overline B)})^{1/3}\right ]^{ { 2\alpha-3 }}.
\end{align}
If  $ f\in W^{1,1}_\loc(\Omega)$  additionally, then  we have
\begin{align}\label{strongest}
  \int_\Omega|D|Du|^{\alpha}|^2\xi^2\,dx
  \le C(\alpha)&\int_\Omega |Du|^{2\alpha}(|D\xi|^2+|D^2\xi||\xi|)\,dx \nonumber\\
  &  \quad\quad+C(\alpha)\left |\int_\Omega f_i u_i|Du|^{2\alpha-4}\xi^2\,dx\right|\quad \forall \xi\in C_c^2(\Omega).
\end{align}

\item[(ii)] For $0<  \alpha\le 3/2$ and $1\le p<   3/(3- \alpha) $,
  we have
 $   |Du|^{\alpha}  \in W^{1,p}_\loc(\Omega) $.

 \item[(iii)]  For   $\alpha>3/2$, we have  $|Du|^{2\alpha-6}\in L^1_\loc(\Omega)$ and
 $$  |Du|^{   2\alpha-6}\le\frac1{\alpha} \frac1{f^2}|D |Du|^{\alpha}|^2  \quad{\rm a.\,e.\ in }\ \Omega.$$

\item[(iv)] We have  \begin{equation}\label{pw}-(|Du|^{\alpha})_iu_i=\alpha|Du|^{{ \alpha-2}}f\quad{\rm a.\,e.\ in}\ \Omega \quad\forall \alpha>0\end{equation}
and
\begin{equation}\label{chain}|Du|^\tau D|Du|^\alpha=\frac{\alpha}{\alpha+\tau}D|Du | ^{ \alpha+\tau }\quad{\rm a.\,e.\ in}\ \Omega\quad \forall  \alpha,\tau>0.
\end{equation}
\end{enumerate}
\end{thm}

Below, we give an example to clarify the sharpness in Theorem \ref{mainthm}. We also state
a Gehring type conjecture on the higher integrality of $|D|Du|^\alpha|$ when $\alpha>3/2$,
and moreover,   fully describe viscosity solutions to 1-dimensional
inhomogeneous $\fz$-Laplace equations.

\begin{rem} \begin{enumerate}
 \item[(i)]
Note that the function $w(x_1,x_2)=-x_1^{4/3} $  satisfies
 $$- \Delta_\fz w=\frac{4^3}{3^4}\ \mbox{in $\rr^2 $}$$     in  viscosity sense.
A direct calculation gives that
$$
 |Dw (x )| = C |x_1|^{1/3}  \ {\rm and}\
   |D|Dw|^\alpha  (x )| =C(\alpha) |x_1|^{-(3-\alpha)/3} \quad\forall x\in\rr^2 \setminus\{0\}.$$

The regularity of $w$ leads to the sharpness in Theorem \ref{mainthm}. Precisely,
    Theorem \ref{mainthm}  (iii) is sharp   in the sense that
     $|Dw|^{2\times ( 3/2)-6}=C |x_1|^{-1} \notin L^1_\loc(\rr^2)$.
   When $\alpha\in(0,3/2]$,
   Theorem \ref{mainthm} (ii) is   sharp   in the sense that
 $|D|Dw|^\alpha|^{3/(3-\alpha)}=C(\alpha)|x_1|^{-1}\notin L^1_\loc(\rr^2)$.
 Theorem \ref{mainthm} (i) is (asymptotic) sharp   in the sense that, for any  $p> 2$,
  $|D|Dw|^\alpha|^p=C(\alpha)|x_1|^{- (3-\alpha)p/3}\notin L^1_\loc(\rr^2)$ whenever    $\alpha\in(3/2,3-3/p)$, that is, $(3-\alpha)p/3\ge 1$.

\item[(ii)] For each fixed $\alpha>3/2$, note that
 $ |Dw|^\alpha \in W^{1,p}_\loc(\rr^2)$ for any $p\in(2,3/(3-\alpha))$ if $\alpha<3$ or for any $p\in(2,\fz]$ if $\alpha\ge3$.
 Comparing with Theorem \ref{mainthm} (i), we pose the following Gehring type conjecture.

\noindent {\bf Conjecture:} Suppose $\Omega\Subset\mathbb R^2$   and  $f\in BV_\loc (\Omega)\cap C^0(\Omega)$ with  $|f|>0$ in $\Omega$.
For each $\alpha>3/2$  there exists some $\ez_\alpha>0$ such that $\ez_\alpha\to0$ and
$|Du|^{\alpha}\in W^{1,2+\ez_\alpha}_\loc(\Omega)$ for all viscosity solutions $u$ to \eqref{inhom infty equ}.

  If this conjecture is true, then
  one would conclude the  $C^1$- and $C^{1,\gz}$-regularity for some $\gz>0$
  of  viscosity solutions  to \eqref{inhom infty equ}, which remains open now.
%
 \item[(iii)] The
 function $w$  given in (i) is essentially of dimension 1. Below we fully describe viscosity solutions to
 inhomogenous $\infty$-Laplace equation in dimension 1:
 \begin{equation}\label{1d}-u'u''u'=f\quad\mbox{ in  $I $.}
 \end{equation}
  where $I\Subset \rr$ is any open interval.

Without loss of generality, let $I=(0,1)$, and  $f\in   C^0(\overline I)$ with $|f|>0$ in $\overline I$.
If $u\in C^0(\overline I)$  is a viscosity solution to \eqref{1d}, then
 $$
u (t)=u(0)+\int_{0}^t\left[\int_{0}^s[-3f(r)]\,dr-c\right]^{1/3}\,ds\quad\forall t\in I,$$
 where $c\in \rr$ is uniquely determined by the value $u(1)$.

 From above formula one can see that $u\in C^{1,1/3} (I)\cap W^{2,p}_\loc(I)\cap C^2(I\setminus I_0)$ with $p\in[1,3/2)$, and $u$ is strictly convex if $f<0$ and strictly concave if $f>0$. Here the set $I_0:=\{t\in I, u'(t)=0\}$ contains at most  one point, and  if  $I_0$ contains some $t_0\in I$, then    \begin{equation*}  \lim_{s\to t_0} \frac{3[u (s)-u(t_0)]}{4(s-t_0)^{4/3}}= \lim_{s\to t_0} \frac{u'(s)}{(s-t_0)^{1/3}}= [-3f(t_0)]^{1/3}.
 \end{equation*}
 In particular, the conjecture in (ii) is true in dimension 1.

 Moreover, $|u'|^{-1}\in L^{p}_\loc(I)\cap   C^1(I\setminus I_0)$ for any $p\in(0,3)$;
 $|u'|^\alpha\in W^{1,p}_\loc(I)\cap C^1(I\setminus I_0)$
whenever $\alpha\in(0,3)$ and $p\in[1,3/(3-\alpha))$;
$|u'|^3\in W^{1,\fz}_\loc(I)\cap C^1(I\setminus I_0)$;  $|u'|^\alpha\in C^1(I)$   whenever
$\alpha>3$. We also have
 $(|u'|^\alpha)' =-\alpha |u' |^{\alpha-4}u'  f $  everywhere in $ I\setminus I_0$ whenever $\alpha\in(0,3]$,
  and   everywhere in $I$ whenever $\alpha\in(3,\fz)$.
  In particular, $-u'' |u' |^2=f  $  everywhere in $ I\setminus I_0$.

%
%
%
%
%
 \end{enumerate}
\end{rem}

Next, we compare   Sobolev regularity  in the case $|f|>0$  with that in  the case $f\equiv0$.
\begin{rem}\label{cpi}
The Sobolev regularity for viscosity solutions to \eqref{inhom infty equ} given in
  Theorem \ref{mainthm} and the sharpness above are very different from that for planar $\infty$-harmonic functions  (that is, in the case $f\equiv0$) by \cite{kzz} as stated above.
When considering   $W^{1,2}_\loc$-regularity for $|Du|^\alpha$, the role of $\alpha=3/2$ in Theorem \ref{mainthm} plays the role of $\alpha=0$ for $\fz$-harmonic functions.   When  $0<\alpha\le 3/2$,  Theorem \ref{mainthm} (ii)\&(iii)  have their own interest, and we have to   treat them separately.  Moreover,
  consider  $\wz w^\ez(x)=x_1^{4/3}-\ez x_2^{4/3}$ with $\ez\in(-1,1)$, which satisfies
  $$-\Delta_\fz \wz w^\ez=(1-\ez^3)\frac{4^3}{3^4}>0\quad\mbox{ in $\rr^2$}$$
 in viscosity sense.  Note that $-\det D^2\wz w^\ez $ is nonnegative when $\ez>0$ and nonpositive when $\ez<0$.
This reveals that the distributional determinant for viscosity solutions to  \eqref{inhom infty equ} may change sign,
 and hence, behave much more complicated than $\fz$-harmonic functions.

\end{rem}

 We also list some relations between (i) to (iv) of Theorem \ref{mainthm}.

\begin{rem}\label{relthm}
\begin{enumerate}
\item[(i)]

 Theorem \ref{mainthm} (iii) follows from Theorem \ref{mainthm} (i)\&(iv); hence, to obtain Theorem \ref{mainthm}, it suffices to prove Theorem \ref{mainthm} (i),(ii)\&(iv).
    Indeed,
by \eqref{pw} with $\alpha=2$ and $|f|>0$ in $\Omega$, we know that $|Du|>0$ a.\,e. in $\Omega$. For any $\alpha>3/2$,
 by \eqref{pw} again, we have
$$  |Du|^{   2\alpha-6}=|Du|^{-2} \frac1{\alpha^2f^2}[(|Du|^{\alpha})_iu_i]^2\le\frac1{\alpha^2} \frac1{f^2}|D |Du|^{\alpha}|^2  \quad{\rm a.\,e.\ in }\ \Omega.$$
 By Theorem \ref{mainthm} (i), we conclude  $|Du|^{   2\alpha-6}\in L^1_\loc(\Omega)$, that is, Theorem \ref{mainthm} (iii).

 \item[(ii)]

  For $0<\alpha\le 3/2$ and $1\le p<3/(3-\alpha)$,   no quantative estimates for $|D  |   Du|  ^{  \alpha  } |^p$ is given in Theorem \ref{mainthm} (ii).
Via  Theorem \ref{mainthm} (i)\&(iv), there is a  pointwise estimate
 for $|D  |   Du|  ^{  \alpha  } |^p$ as follows:
 letting $ \bz\in(3/2,3/p-3/2+\az)$, by $3+(\az-\bz)p/(2-p) > 3/2$ and H\"older's inequality we have
  \begin{align*}
|D  |   Du|  ^{  \alpha  } |^p & = (\beta/2)^p |Du|^{(\az-\beta)p}|D|Du|^\beta|^p\\
&\le
C(\az,\bz,p) [|D|Du|^\bz|^2+|Du|^{2(\az-\bz)p/(2-p)}] \\
&\le
C(\az,\bz,p) [|D|Du|^\bz|^2+\frac1{f^2} |D|Du|^{(\az-\bz)p/(2-p)+3}|^2]\quad \mbox{ a.\,e. in $\Omega$. }
\end{align*}

\end{enumerate}
\end{rem}

Now we sketch the ideas  for the proof of   Theorem \ref{mainthm}.  Up to considering $-u$ and $-f$,  in the sequel we always assume $f\in BV_\loc(\Omega)\cap C^0(\Omega)$ and $f>0$ in $\Omega$.
Given arbitray $U\Subset\Omega$, write  $\ez_U=\frac14\min\{\dist(U,\partial \Omega),1\}$, and
let  $ f^\ez\in C^\fz(  U) $  with $\ez\in(0,\ez_U]$ be the standard smooth mollifications of $f$.

In Section 3, as motivated by Evans (see \cite{e03,ey04,es11a,es11b,kzz}) in the case  $ f\equiv 0$ and   by \cite{l,lmz}  in  the case $ f\in C^{0,1}(\Omega)$,
we consider the following approximation to equation \eqref{inhom infty equ}:
 For $\ez\in(0,\ez_U]$,  let $ u^\ez\in C^\fz(U)\cap C(\overline U)$ be a solution to the equation
$$-\Delta_\fz u^\ez-\ez\Delta u^\ez=f^\ez\ {\rm in }\ U;\ u^\ez|_{\partial U}=u|_{\partial U}.$$
Recall that a uniform $C^0(\overline U)$-estimate  and a uniform boundary regularity estimate  for $u^\ez$ were   established in \cite{l,lmz}; see Lemma \ref{maxandboundary}.
Assuming $f\in W^{1,q}_\loc(\Omega)$ with $q\in(1,\fz]$ in additional, and
 observing in  Lemma \ref{identity1} the crucial
 identity
\begin{equation}\label{identity}|D^2u^\ez  Du^\ez|^2= - \ez  (\Delta u^\ez)^2-f^\ez\Delta u^\ez -|Du^\epsilon|^2\det D^2u^\ez \quad {\rm \ in}\ U
\end{equation}
(see  also \cite{kzz} when $f\equiv0$),
we   establish the following  uniform Sobolev estimates in Section 3.
\begin{enumerate}
\item[$\bullet$] By \eqref{identity}, we show in Lemma \ref{Duez} that,  for any ball $ B\Subset U$, the
   $L^2(B)$-norms  of $|D|Du^\ez|^2| |u^\ez|+|u^\ez|^3$ are uniform bounded in $\ez>0$; see Section 6 for the proof.
Together with Sobolev's imbedding and $f\in W^{1,q}_\loc(\Omega)$, this
 implies that for any $p\in[1,\fz)$,  $u^\ez\in W^{1,p}_\loc (U)$ uniformly in $\ez>0$; see Lemma \ref{uniform Sobolev}.
When $q=\fz$, it was proved in \cite{l,lmz} that   $u^\ez\in W^{1,\fz}_\loc(U)$ uniformly in $\ez>0$,
 which is still unavailable when $q<\fz$, see Remark \ref{sobolevbv}.
\item[$\bullet$]
By \eqref{identity}, we establish  some
Sobolev estimates for $|Du^\ez|^\alpha$
or $(|Du^\ez|^2+\kz)^{\alpha/2}$ which are uniform  in $\ez>0$.
Precisely,  when $\alpha\in \{2\}\cup[3,\fz)$  we  show that
$W^{1,2}(B)$-norms  of  $|Du^\ez|^{\alpha}$ are bounded
  in terms of $L^2 (2B)$-norms of themselves, integral of $f^\ez_i u^\ez_i  |Du^\ez|^{2\alpha-3}$
and some error terms; while when $\alpha\in(3/2,2)\cup(2,3)$,  for $\kz>0$,
similar  $W^{1,2}_\loc$-estimates for $(|Du^\ez|^2+\kappa)^{\alpha/2}$ are established;
see   Lemma \ref{DDuez} whose proof is given in Section 6.
Together with $f\in W^{1,q}_\loc(\Omega)$, we show in Lemma \ref{uniform DDuez} that
when $\alpha\in \{2\}\cup[3,\fz)$,
   $ |Du^\ez|^{\alpha}\in W^{1,2}_\loc(U)$ uniformly in $\ez>0$;
 when $\alpha\in(0,2)\cup(2,3)$, for each $\kappa>0$,
 $|Du^\ez|^2+\kappa)^{\alpha/2}\in  W^{1,2}_\loc(U)$  uniformly in $\ez>0$;
     when   $\alpha\in(3/2,2)\cup(2,3)$,    for any  $V\Subset U$;
 $ \limsup_{\ez\to0}\|D(|Du^\ez|^2+\kappa )^{\alpha/2}\|_{L^{ 2}(V)}$ is uniformly bounded in $\kz\in(0,1)$;
when $\alpha\in(0,3/2]$ and $p\in[1,3/(3-\alpha))$,   for   any $V\Subset U$,
 $ \limsup_{\ez\to0}\|D(|Du^\ez|^2+\kappa )^{\alpha/2}\|_{L^{ p}(V)}$ is uniformly bounded in $\kz\in(0,1)$.

 \item[$\bullet$] By  \eqref{identity}, we establish
 an {\it integral flatness} for $u^\ez$, see Lemma \ref{flat} whose proof is given in Section 6.
 This is crucial to clarify the pointwise limit of $|Du^\ez|^2 $ as $ \ez\to0$ in Section 4.
Here and below, by an integral flatness for $v$ we mean that
for any linear function $P$, the $L^2(B)$-norm  of $\langle Dv, Dv-DP\rangle|Dv|^3$ are controlled by
 $L^2(2B)$-norm  of $|v-P|^2$  times some extra terms (say  $L^2(2B)$-norm of $D|Du^\ez|^2$ and local integration of $f^\ez_i u^\ez_i |Du^\ez|^4$ in the case $u^\ez$).
\end{enumerate}

In Section 4, we prove Theorem \ref{mainthm} and an integral flatness  for $u$ when $f\in W^{1,q}_\loc(\Omega)$ with $q\in(1,\fz]$ additionally.
To this end, we  derive the following  crucial convergence properties 
from uniform Sobolev estimates in Section 3.
\begin{enumerate}

\item[$\bullet$] We first derive  $u^\ez\to u$ in $C^0(U)$ as $\ez\to0$ in Lemma \ref{uniform conv} from
  $u^\ez\in W^{1,p}_\loc(U)$ uniformly in $\ez>0$  by Lemma \ref{uniform Sobolev},    the
uniform boundary estimates  in \cite{l,lmz} and the uniqueness  in  \cite{lw}.

 \item[$\bullet$]   We show in Lemma \ref{strong converge} that,
  as $\ez\to0$, $|Du^\ez|^2\to |Du|^2$ in  $L^{p}_\loc(U)$ and weakly in $W^{1,2}_\loc(U)$, and $u^\ez\to u$ in $W^{1,p}_\loc(U)$ for any $p\in[1,\fz)$.
  Since $|Du^\ez|^2 \in W^{1,2}_\loc(U)$  uniformly in $ \ez>0$ as given by
  Lemma \ref{uniform DDuez},  we  know that $|Du^\ez|^2$
    converges to   some function $h$ in $L^{p}_\loc(U)$ for all $p\ge1$ and weakly
  in $W^{1,2}_\loc(U)$ as $\ez\to0$ (up to some subsequence).
Via the integral flatness for $u^\ez$ given by Lemma \ref{flat}, and  some careful but tedious analysis around  Lebesgue points,  we   prove that $|Du|^\alpha =h$ almost everywhere, and hence
   $u^\ez \to u$ in $W^{1,p}_\loc(U)$ for all $p\ge1$.

 \item[$\bullet$]   Moreover, when  $\alpha\ge3$, since $|Du^\ez|^\alpha \in W^{1,2}_\loc(U)$  uniformly in $ \ez>0$ as given by
  Lemma \ref{uniform DDuez}, by $u^\ez\to u$ in $W^{1,p}_\loc(U)$ for all $p\in[1,\fz)$ as $\ez\to0$,
  we show in Lemma \ref{strong converge} that
  $|Du^\ez|^\alpha\to |Du|^\alpha$  in $L^{p}_\loc(U)$ for all $p\ge1$ and weakly
  in $W^{1,2}_\loc(U)$ as $\ez\to0$.
   Similarly, by  Lemma \ref{uniform DDuez}  we also show
   in Lemma \ref{strong converge} that when $\alpha\in(0,2)\cup(2,3)$ and $\kz>0$,
$(|Du^\ez|^2+\kappa)^{\alpha/2}\to  (|Du |^2+\kappa)^{\alpha/2}$ in $L^p_\loc(U)$ for all $p\ge1$ and weakly in $W^{1,2}_\loc(U)$ as $\ez\to0$;
 when $\alpha\in(3/2,2)\cup(2,3)$,
$(|Du |^2+\kappa)^{\alpha/2}\to   |Du | ^{\alpha }$ in $L^p_\loc(U)$ for all $p\ge1$ and weakly in $W^{1,2}_\loc(U)$ as $\ez\to0$;
  when $\alpha\in(0,3/2]$ and $p\in(1,3/(3-\alpha))$,
$(|Du |^2+\kappa)^{\alpha/2}\to   |Du | ^{\alpha }$ in $L^t_\loc(U)$ for all $t\ge1$ and weakly in $W^{1,p}_\loc(U)$ as $\ez\to0$.
 \end{enumerate}
With the aid of Lemma \ref{strong converge},  we are able to conclude Theorem \ref{mainthm}  from
Sobolev estiamtes of $|Du^\ez|^\alpha $  or $(|Du^\ez|^2+\kz)^{\alpha/2}$ given in Lemma \ref{DDuez}.
 From  Lemma \ref{DDuez} again,  Sobolev convergence in Lemma \ref{strong converge}
  and the integral flatness for $u^\ez$ in Lemma \ref{flat}, we also deduce an integral flatness for $u$  in Lemma \ref{flatu}.

 In Section 5,  we prove Theorem \ref{mainthm} when $f\in BV_\loc(U)\cap C^0(U)$.
In this case,   the above approach fails 
since  the uniform $W^{1,\fz}_\loc(U)$-estimates of  $u^\ez$
  is unavailable as  indicated by Remark \ref{sobolevbv}.
Note that we do need the uniform $W^{1,\fz}_\loc(U)$-estimates of  $u^\ez$
 to obtain  uniform   $W^{1,2}_\loc(U)$-estimates of $ |Du^\ez|^\alpha $ and hence  to prove Theorem \ref{mainthm}; for example, since $f^\ez$ only have uniform $BV(U)$-estimates, we need
the uniform $W^{1,\fz}_\loc(U)$-estimates of  $u^\ez$  to get uniform estimates of the term
 $\int_U f^\ez_i u^\ez_i |Du^\ez|\xi^2\,dx$  in   Lemma \ref{DDuez},
  and hence, to obtain the uniform $W^{1,2}_\loc(U)$-estimates of  $|Du^\ez|^2$.  Therefore, new ideas are required.

Instead of the above approach, we consider an approximation by  $\fz$-Laplace equations with smooth inhomogeneous terms.
That is,
 for each $\dz\in(0,\ez_U]$,
 let $\hat u^\dz$ be the viscosity solution to
  the approximation equations
\begin{equation*}\label{appin}
-\Delta_{\infty} \hat u^\dz =f^\dz \ \text{in $U$}, \quad \hat u^\dz=u   \ \text{on $\partial U$}.
\end{equation*}
Since $f^\dz$ is smooth, as proved in Section 4, Theorem \ref{mainthm}  and also the flatness in Lemma \ref{flatu} hold for $\hat u^\dz$.
Moreover, by Lemma \ref{lipu}, we have $\hat u^\dz\in C^{0,1}(U)$ uniformly in $\dz>0$.

Recall that, as proven by \cite{lw},   $\hat u^\dz\to u $ in $C^0(\overline U)$ as $\dz\to0$.
Since \eqref{strongest} holds for $\hat u^\dz$,
by estimating   $\int_U\hat u^\dz _i  f^\dz_i |Du|^{2\alpha-3}\xi^2\,dx$
 via $C^{0,1}(U)$-norms of $\hat u^\dz$ and $BV(U)$-norms of $f^\dz$, for any $\alpha>3/2$
we conclude uniform  $W^{1,2}_\loc(U)$-estimates of $|D\hat u^\dz|^{ \alpha}$.
By this  and the integral flatness of $\hat u^\dz$ as given by Lemma \ref{flatu},
we are able to show that  $|D\hat u^\dz|^\alpha  \to |Du|^\alpha$ and
 $\hat u^\dz\to u $ in $W^{1,p}_\loc(U)$ for all $p\ge 1$ as $\dz\to0$, see Lemma \ref{strong stability}.
This allows us to conclude Theorem \ref{mainthm}  from \eqref{strongest}, uniform $C^{0,1}(U)$-estimates of $\hat u^\dz$ and
uniform $BV(U) $-estimates of $f^\dz$. Theorem \ref{mainthm} then follows.

Finally we make some convention.
Denote by $C$ an absolute constant (independent of main parameters) and
 by $C(a,b,\cdots)$ a constant depending the parameters $a, b,\cdots$.
Write $B(x,r)$ for a ball centered at $x$ and with radius $r>0$, $\overline {B(x,r)}$ as the closure of $B(x,r)$,
 and  $CB(x,r)=B(x,Cr)$ for $C>0$.
The notation $V\Subset U$ means that $\overline V$ is compact and $\overline V\subset U$.
We write $\dist (x,  F)=\inf_{y\in F} |x-y|    $ and $\dist (E,  F)=\inf_{x\in E}\dist(x,F)$.
Denote by $C^0(E)$   the collection of continuous functions on
a set $E\subset \mathbb R^2$.
For $k\ge 1$,   $C^k(U)$  consists of functions $u $  on an  open set $U\subset \mathbb R^2$ such that $Du\in C^{k-1}(U)$;
  $C^\fz(U):=\cap_{k\in\nn} C^k(U)$.   Write $C^k_c(U)$  denotes the class of functions  in $C^k (U)$ which compactly supported in $U$.
  For $ k\in \nn\cup\{0\}$ and $\gamma\in(0,1]$, $C^{k,\gz}(U)$ denotes the collection of H\"older continuous function of order $\gz$.
For $p\ge 1$, $L^p(U)$ denotes the $p$-th integrable Lebesgue space; $L^\fz(U)$ as the space of essentially bounded functions.
For $1\le p\le \fz$, $L^p_\loc(U)$ is the collection of functions $v$ such that $v\in  L^p(V)$ for all $V\Subset U$.
For $1\le p\le \fz$, $W^{1,p}(U)$ is the first order $p$-th Sobolev space, that is, the set of functions $v$ on $U$ whose distribbutional derivatives $Dv\in L^p(U)$;
similarly define $W^{1,p}_\loc(U)$.
 We also write  $W^{1,\fz}(U)$ as $C^{0,1}(U)$.  

\section{Some facts for inhomogeneous $\fz$-Laplace equations}

We recall several facts about the inhomogeneous $\fz$-Laplace equation.
Suppose that $f\in  C^0(\Omega)$, and let $ u$ be a  viscosity solution to
$-\Delta_{\infty} u  =f  \ \text{in $\Omega$}.$
Up to considering $-u$ and $-f$, we may assume that $f>0$.
Notice that $u-a$ for arbitrary $a\in\rr$ is also a viscosity solution.
See \cite{lw} for the following maximum principle (Lemma \ref{max}),  uniqueness (Lemma \ref{unique}), and  stability (Lemma \ref{stability}).
\begin{lem}\label{max}
 For any $U\Subset \Omega$, we have
$$\max_{\overline U} |u|\le C(\|u\|_{C^0(\partial U)}, \|f\|_{C^0(\overline U)}).$$
\end{lem}

\begin{lem}\label{unique}
Let  $U\Subset \Omega$ and assume $|f|>0$ in $U$. If $v\in C(\overline U)$ is a viscosity solution to
\begin{equation*} -\Delta_{\infty} v  =f  \quad \text{in $U$};\ v=u\quad \text{on $\partial U$},
\end{equation*}
then $v=u$ in $\overline U$.
\end{lem}

\begin{lem}\label{stability}
Let  $U\Subset \Omega$ and assume $|f|>0$ in $\overline U$.
For   $\dz\in(0,1]$, let $f^\dz\in C^0(\overline U)$ such that $f^\dz\to f$ in $C^0(\overline U)$, and let $\hat u^\dz\in  C^0(\overline U)$  be
a viscosity solution to
$$\Delta_\fz \hat u^\dz=f^\dz \ {\rm in}\ U; \quad \hat u^\dz=u \ {\rm on}\ \partial U.$$
Then $\hat u^\dz\to u$ in $C^0(\overline U)$.
\end{lem}

Moreover,
it is known that   $u\in C^{0,1}(\Omega)$, see \cite{lw,l}.
The following quantative estimates  essentially follow from   \cite{l}.
\begin{lem}\label{lipu} For any ball $B\subset 2 B\Subset \Omega$ with radius $R$, we have
$$ \|u\|_{C^{0,1}(B)}\le   C\frac1R\|u\|_{C^0(2 \overline B )}+C (R\|f\|_{C^0(2 \overline B )})^{1/3}.$$
\end{lem}
\begin{proof}
Up to some translation and scaling, we may assume that $B=B(0,1/2)$, and it suffices to prove that
$$\|u\|_{C^{0,1}(B)}\le   C \|u\|_{C^0(2 \overline B )}+C  \|f\|_{C^0(2 \overline B )} ^{1/3}.$$
Consider the function
$$ \wz u  (x,x_{3})=  \frac{  u(x)}{4^{1/3} \|f\|^{1/3}_{C^0(2\overline B)}}+ 5x_{3} $$ on $2B\times \rr$.
Note that
$\wz \Delta_\fz \wz u =\wz f $ in $2B\times \rr$, where $\wz \Delta_\fz$ is the $3$-dimensional $\fz$-Laplacian
and $ |\wz f (x,x_3)|= |f(x)|/4  \|f\|_{C^0(2\overline B)}<1/2 $.

Note that for each $\wz x\in \wz B=B((0,0),1/2)$ and each $r<1-|\wz x|$, we have  $ \pm L^\pm_r(\wz u, \wz x )\ge 5$,
where $$L^+_r(\wz u, \wz x ) :=\sup_{\partial B(\wz x,r)} \frac{\wz u(\wz y)-\wz u(\wz x)}r\quad {\rm and}\quad L^-_r(\wz u, \wz x ): =\inf_{\partial B(\wz x,r)} \frac{\wz u(\wz y)-\wz u(\wz x)}r.$$
As proved in \cite[Corollary 1]{l},
for $\wz x\in \wz B $ the function $r\in(0,1/2)\to \pm L^{\pm}_r(\wz u,\wz x)+r $  is increasing.
Thus for $\wz x\in \wz B $,
$$|
D\wz u(\wz x)|\le \max\{L^+_{1/2} (\wz u,\wz x),-L^-_{1/2} (\wz u,\wz x)\}+\frac12.$$
This yields
\begin{align*}\| \wz u\|_{C^{0,1}(\wz B)}
  &\le     \|\wz u \|_{C^0(2\overline {\wz B})}+\frac12\le  4^{-1/3}\|f\|^{-1/3}_{C^0(2\overline B )}  \| u \|_{C^0(2 \overline B)}+ 11,
  \end{align*}
  which further implies
  $$\| u\|_{C^{0,1}(B)}\le 4^{1/3}\|f\|^{1/3}_{C^0(2\overline B )} \| \wz u \|_{C^{0,1}(\wz B) }\le   C \|u\|_{C^0(2\overline B )}+C  \|f\|_{C^0(2\overline B )} ^{1/3} $$
  as desired.
 \end{proof}

\section{Uniform  estimates for approximation equations when $ f\in (\cup_{q>1}W^{1,q}_\loc(\Omega))\cap C^0(\Omega)$}

Suppose $\Omega \Subset\mathbb R^2$, and $ f\in W^{1,q}_\loc(\Omega)\cap C^0(\Omega)$ for some $q>1$ with $f>0$ in $\Omega$.
Let $U \Subset\Omega$   and $\ez_U:=\min\{\frac14\dist(U,\partial\Omega),1\}$.
For each  $ 0<\ez<\ez_U$, write
$$
  f^\ez(x)=\int_\Omega f(x-z)\frac1{\ez^2}\vz(\frac z{\ez})\,dz\quad  \forall x\in U, $$
where $\vz\in C_c^\fz(B(0,1))$, $0\le \vz\le 1$ and $\int_{\mathbb R^2}\vz(z)\,dz=1$.   The following simple facts are used quite  often:
for all $\ez\in[0,\ez_U]$ and $\wz U=\{x\in\Omega, \dist(x,\partial \Omega)>2\ez_U \}$,
$$ \|f^\ez\|_{C^0(\overline U)}\le  \|f\|_{C^0(\overline {\wz U})}\quad{\rm and }\quad\|  f^\ez\|_{W^{1,q}(U)}\le  \|f\|_{W^{1,q}(\wz U)};$$
moreover, for all $B=B(x,R)\Subset U$ and $\ez<\min\{R,\ez_U\}$,
$$ \|f^\ez\|_{C^0(\overline B)}\le  \|f\|_{C^0(2\overline B)}\quad{\rm and }\quad  \| f^\ez\|_{W^{1,q}(B)}\le  \|f\|_{W^{1,q}(2B)}.$$

For each  $  \ez \in(0,\ez_U]$,  let $ u^\epsilon\in C^\infty(U)\cap C(\overline U)$ be a   solution to
\begin{equation}\label{infty ezf}
-\Delta_{\infty} u^\epsilon -\ez \Delta u^ \epsilon=f^\ez \quad \text{in $U$};\quad u^\ez=u   \quad \text{on $\partial  U$};
 \end{equation}
 see for example \cite{lmz} for the existence of such $u^\ez$.
The following  uniform estimates  and  boundary uniform estimates of $u^\ez$ follows from \cite{l,lmz}.

\begin{lem}\label{maxandboundary}
We have
$$
\sup_{\ez\in(0,\ez_U]}\max_{\overline{ U}} |u^{\ez}| \le C (\|u\|_{C(\partial U)},\|f\|_{C^0(\overline {\wz U})});
$$
and there exists $\gamma\in(0,1)$ such that
$$
\sup_{\ez\in(0,\ez_U]} \big|u^\epsilon(x)-u(x_0)\big|\le C(\|f\|_{C^0(\overline {\wz U})})|x-x_0|^\gamma, \ \forall\ x\in U, \ x_0\in\partial U.
$$
\end{lem}

The following identity is crucial to establish  uniform Sobolev estimates of $u^\ez$ and $|Du^\ez|^\alpha$.
\begin{lem}\label{identity1} For each $\epsilon\in(0,\ez_U)$ we have
\begin{equation}\label{key  identityII}(-\det D^2u^\ez) |Du^\ez|^2=  |D^2u^\ez Du^\ez|^2 + \ez(\Delta u^\ez)^2+f^\ez\Delta u^\ez \quad {\rm in}\ U.\end{equation}
\end{lem}

\begin{proof}
The equality \eqref{key identityII} follows  from $-\Delta_\infty u^\ez= \ez \Delta u^\ez+f^\ez$
and the following equality
$$
(-\det D^2v) |Dv|^2= |D^2v Dv|^2  -\Delta v\Delta_\infty v  \quad \forall v\in C^\fz(U) .
$$
This equality was observed in \cite{kzz};
 the details is given  for  reader's convenience as below:
\begin{align*}
 |D^2v Dv|^2
&= (v_1v_{11}+v_2v_{12})^2+ (v_1v_ {21}+v_2v_{22})^2\\
&=v_{11} [(v_1)^2v_{11}+ 2v_1v_2 v_{12}]+ v_{22} [(v_2)^2v_{22}+ 2v_1v_2 v_{12}]+  [v_{12})^2((v_1)^2+(v_2)^2]\\
&=(v_{11}+v_{22} ) \Delta_\infty v-  v_{11}v_{22}[(v_2)^2+(v_1)^2]+(v_{12})^2[(v_1)^2+(v_2)^2]\\
&=\Delta v\Delta_\infty v +(-\det D^2v) |Dv|^2.
\end{align*}
 This completes the proof of Lemma \ref{identity1}.
\end{proof}

Associated to such $u^\epsilon$, we introduce a functional $\mathbb I_\epsilon$ on  $ C_c (U)$ defined by
\begin{align*}
\mathbb I_\epsilon(\phi)
&= \int_{U} -\det D^2u^\ez \phi \, dx\quad\forall \phi\in C_c(U).
\end{align*}
By  Lemma \ref{identity1} we   write
\begin{align}
\mathbb I_\epsilon(\psi |Du^\ez|^2)
&= \int_{U}  |D^2u^\ez Du^\epsilon |^2 \psi \, dx +\ez\int_{U}   (\Delta u^\ez)^2  \psi \, dx+  \int_{U}    f^\ez \Delta u^\ez   \psi \, dx \quad \forall  \psi\in  C_c (U). \label{identityIII}
\end{align}
In particular,   we have $$\int_{U} -\det D^2u^\ez |Du^\ez|^2\psi \, dx\ge \int_{U}    f^\ez \Delta u^\ez   \psi \, dx\quad \forall 0\le
 \psi\in  C_c (U).$$

On the other hand,
 for any  $v\in C^\fz(U)$  the determinant $\det D^2v$ is actually of  divergence form, that is,
$$
-\det D^2v
=-\frac12{\rm div}(\Delta v Dv-D^2vDv)
\quad{\rm in}\ U.
$$
We  further write
\begin{align}\label{functional}\mathbb I_\epsilon (\phi) &= \frac12\int_{U} [\Delta u^\ez u^\ez_i\phi_i  - u^\ez_{ij} u^\ez_j \phi_i]\,dx
\quad \forall  \phi\in W^{1,\,2}_c (U) \end{align}

Letting  $\phi=|Du^\ez|^2   (u^\ez\xi)^2$ in \eqref{functional}  we obtain the following estimates. The proof is  postponed to Section 6.

\begin{lem}\label{Duez}
For  any  $\xi\in C_c^2(U)$, we have
\begin{align*}
 &\int_{U}  |D^2u^\ez Du^\epsilon |^2   ( u^\ez)^2\xi^6\, dx
 +\int_{U}  |Du^\ez| ^6\xi^6 \,dx\\
 &\quad\le  C
\int_{U} |u^\ez |^6(| D\xi|^2+|D^2\xi||\xi|)^3\,dx+
C
\int_{U} \xi^6 (f |u^\ez| )^{3/2} \,dx +C \left|\int_{U}          f^\ez_i u^\ez_i   (u^\ez)^2   \xi^6    \,dx\right|\\
&\quad\quad+C  \ez^3\int_{U}   \xi^6\,dx +C   \ez^{3/2} \int_{U} |D\xi|^{  \alpha +1} ( u^\ez)^3 \xi^3 \,dx.
 \end{align*}
\end{lem}

From Lemma \ref{Duez}  and the Sobolev imbedding we conclude the following uniform local Sobolev estimates of $u^\ez$.
\begin{lem}\label{uniform Sobolev} For  each $1\le p <\fz$, we have $ u^\ez\in W^{1,p}_\loc(U)$ uniformly in $\ez>0$, and moreover,
\begin{equation}
\label{uniformDuez}
\sup_{\ez\in(0,\ez_U]}\|Du^\ez\|_{L^p(B)}\le C(p, q,\osc_{2B}u,\|f\|_{C^0(2\overline B)},B,  \|Df\|  _{L^q(2B)})\quad\forall B\subset 2B\Subset U.
\end{equation}
\end{lem}

\begin{proof}
By the H\"older inequality, it suffices to consider all $p$ sufficiently large such that $2p/(2p-1)\le q$.
Up to considering $u^\ez -a$ and $u-a$ for $a\in \rr$, we may assume that $1\le u^\ez \le M$ for all $\ez\in(0,\ez_U]$.
By Lemma \ref{Duez},
for any $0\le \xi\in C_c^2(U)$ we have
\begin{align*}
 &\int_{U}  |D^2u^\ez Du^\epsilon |^2   \xi^6 \, dx+
 \int_{U} |Du^\ez|^{6} \xi^6 \,dx\\
 &\quad\le  C M^6
\int_{U}  (|D \xi|^6+|D^2\xi|^3\xi^3)\,dx + CM^2\int_{U}
    |D u^\ez|    |Df^\ez|   \xi^6   \, dx + C\ez^2 \int_U|D\xi|^2\,dx.
 \end{align*}
Therefore for any ball $B\Subset 2B\Subset  U$, let $\xi$ be a cut-off function supported in $2B$ such that $\xi=1$ on $B$,
$|D\xi|\le \frac C R$ and $|D^2\xi|\le \frac C{R^2}$, where $R$ is the Radius of $B$.
We obtain
$$\int_{U}  |D^2u^\ez Du^\epsilon |^2\xi^6    \, dx+
 \int_{U} |Du^\ez|^{6} \xi^6 \,dx \le  C (\|f\|_{L^\fz(U)}, M,B )+C(M) \int_{U}
    |D u^\ez|    |Df^\ez| \xi^6     \, dx  .$$
By Sobolev's imbedding, we obtain
\begin{align*}
&\left[\int_{2B}(|Du^\epsilon |^ 2\xi^3 )^p \,dx\right]^{2/p}\\
&\quad\le C(p,B)  \int_{2B}  |D (| Du^\epsilon |^2\xi^3)|^2    \, dx   \\
&\quad\le C(p,B) \int_{2B}  |D  | Du^\epsilon |^2|^2 \xi^6   \, dx   +C(p,B)  \int_{2B}  |D \xi|^2 | Du^\epsilon |^4\xi^4    \, dx   \\
 &\quad\le C(p,B) \int_{2B}     |D  ^2u^\ez Du^\epsilon |^2|^2 \xi^6   \, dx   +C(p,B)  \int_{2B}   |Du^\epsilon |^6\xi^6    \, dx  +C(p,B)
\\
&\quad\le   C (p,\|f\|_{C^0(\overline U)}, M,B )+C(p,M) \int_{2B}
    |D u^\ez|    |Df^\ez| \xi^6     \, dx.
 \end{align*}
Since \begin{align*} \int_{2B}
    &|D u^\ez|    |Df^\ez| \xi^6     \, dx  \\
    &\le  \left [\int_{2B}    |D u^\ez| ^{2p}   \xi^{3p}     \, dx\right]^{1/2p} \left[ \int_{2B}    |Df^\ez| ^{(2p-1)/2p}      \, dx\right ]^{2p/(2p-1)}\\
&\le \frac12 \left[\int_{2B}
    |D u^\ez| ^{2p}   \xi^{3p}     \, dx\right]^{2/ p} + C(p,M) \left[ \int_{2B}
  |Df^\ez| ^{2p/(2p-1)}      \, dx \right]^{2(2p-1)/3p},
 \end{align*}
by $2p/(2p-1)<q$ we arrive at
\begin{align*}
\left[\int_{2B}(|Du^\epsilon |^ 2\xi^3)^p \,dx\right]^{2/p}
&\le   C (p,f, M,B )+ C(p,M) \left[ \int_{2B}
  |Df^\ez| ^{2p/(2p-1)}      \, dx \right]^{2(2p-1)/3p}\\
  &\le   C (p,f, M,B )+ C(p,q,M,B)\left[ \int_{2B}
  |Df^\ez| ^q      \, dx \right]^{4/3q}\\
    &\le   C (p, q, \|f\|_{C^0(\overline U)}, M,B,\|D f\|_{L^q(2B )} ).
 \end{align*}
This finishes the proof of Lemma \ref {uniform Sobolev}.
\end{proof}

\begin{rem}\label{sobolevbv}
\begin{enumerate}
\item[(i)] Under $f\in W^{1,\fz}_\loc(\Omega)$, it was proved by \cite{l,lmz} via the maximal principle that $u^\ez\in W^{1,\fz}_\loc(U)$ uniformly in $\ez>0$; see \cite{es11a} for the case $f\equiv0$.  This implies Lemma \ref{uniform Sobolev}.
But when $f\notin W^{1,\fz}_\loc(\Omega)$, the approach
  in \cite{es11a,l,lmz} via maximal principle fails. 

\item[(ii)] Under $f\in W^{1,q}_\loc(\Omega)\cap C^0(\Omega)$ with $q\in(1,\fz)$, Lemma 3.4 only gives the uniform $W^{1,p}_\loc(U)$-estimates of  $u^\ez $ for each $1\le p<\fz$ but not $p=\fz$.
When $p=\fz$, the approach in Lemma 3.4 fails since the Sobolev imbedding $W^{1,2}_\loc\to L^\fz_\loc(U)$  fails.

\item[(iii)] Under $f\in  W^{1,1}_\loc(\Omega)\cap C^0(\Omega)$ or $f\in BV_\loc(\Omega)\cap C^0(\Omega)$,
  for any given $p\in[1,\fz]$, the uniform $W^{1,p}_\loc(U)$-estimates of  $u^\ez $   is still unavailable.
The approach in Lemma 3.4 fails. Indeed, since
  $f^\ez$ only have uniform $W^{1,1} (U)$- or $BV(U)$-estimates,  the uniform  $W^{1,\fz}_\loc(U)$-estimates of  $u^\ez $
  is required   to get uniform estimates for the term
  $\int_{U}          f^\ez_i u^\ez_i   (u^\ez)^2   \xi^6    \,dx$     in Lemma \ref{Duez}.
  But the failure of the Sobolev imbedding $W^{1,2}_\loc\to L^\fz_\loc(U)$ does not allows to control
 $W^{1,\fz}_\loc(U)$-norms of  $u^\ez $ uniformly.

   \item[(iv)] Considering Lemma \ref{lipu}, we expect
     that under only $f\in    C^0(\Omega)$ one would have $u^\ez\in W^{1,\fz}_\loc(U)$ uniformly  in $\ez>0$.
     To prove this, new ideas are definitely  required.
     \end{enumerate}
\end{rem}

Letting $\phi=|Du^\ez|^2 (|Du^\ez|^2+\kappa)^{\alpha/2} \xi^2$  in  \eqref{functional}  we have the  following
  Sobolev estimates for $|Du^\ez|^{\alpha}$  or $(|Du^\ez|^2+\kappa)^{\alpha/2}$ when  $ \alpha>3/2$. The proof is postponed to Section 6.

\begin{lem}\label{DDuez}
Let $\xi\in C_c^2(U)$ and $\kappa>0$.
\begin{enumerate}
\item[(i)] If $  \alpha =2$, then
\begin{align*}
& \int_U  |D^2u^\ez Du^\epsilon  |^2 \xi^2  \,dx  +
 {\ez}   \int_U  (\Delta u^\ez)^2   \xi^2   \,dx \\
 &\quad\le C    \int_{U}|Du^{\ez}| ^4 (|D\xi|^2+|D^2\xi||\xi|) \,dx + C\left|\int_{U}    u^\ez_i   f^\ez_i \xi^2  \, dx\right| +C  \ez^2 \int_{U}     |D\xi|^2       \, dx.
\end{align*}

\item[(ii)] If $  \alpha \ge3$, then
\begin{align*}
& \int_{U}  |D^2u^\ez Du^\epsilon  |^2|Du^{\ez}| ^{2\alpha-4}  \xi^2 \,dx  +
 {\ez}   \int_U  (\Delta u^\ez)^2 |Du^{\ez}| ^{2\alpha-4}      \xi^2\,dx\\
&\quad\le    C(\alpha) \int_{U}|Du^{\ez}| ^{ 2\alpha }(|D\xi|^2+|D^2\xi||\xi|) \,dx
 +  C(\alpha) \left |\int_{U}        u^\ez_i f^\ez_i   |Du^{\ez}| ^{2\alpha-4} \xi^2\, dx\right|\\
  &\quad\quad+   C(\alpha) \ez\left [\int_{U}     (\Delta  u^\epsilon)^2       \xi^2  \, dx\right]^{1/2}   \left[\int_{U}     (f^\ez )^2 |Du^\epsilon |^{4  \alpha -12}      \xi^2  \, dx\right]^{1/2}.
\end{align*}
\item[(iii)] If $2< { \alpha }<3$, then
\begin{align*}
& \int_{U}  |D^2u^\ez Du^\epsilon |^2  (|Du^\ez|^2+\kappa)^{   { \alpha-2}}  \xi^2 \, dx
+   \ez\int_{U}  | \Delta u^\epsilon |^2  (|Du^\ez|^2+\kappa)^{{ \alpha-2}} \xi^2 \, dx\\
&\quad\quad\quad+ \int_{U}
   ( f^\ez)^2  (|Du^\ez|^2+\kappa)^{   \az-3}  \xi^2  \, dx\\
&\quad\le C ({ \alpha})   \int_{U}(|Du^{\ez}|^ 2+\kappa)^{ \alpha}(|D\xi|^2+|D^2\xi||\xi|)  \,dx
 +  C ({ \alpha}) \left|\int_{U}         u^\ez_i f^\ez_i  (|Du^{\ez}| ^2+\kappa) ^{ { \alpha-2}} \xi^2\, dx\right|\\
  &\quad\quad+   C(\alpha)\ez \kappa^{ \alpha-3}\left[\int_{U}     (\Delta  u^\epsilon)^2       \xi^2  \, dx\right]^{1/2}  \left [\int_{U}     (f^\ez )^2      \xi^2  \, dx\right]^{1/2}\\
&\quad\quad +C(\alpha)\kappa^{{ \alpha-3/2}} \int_{U}(
   |D  f^\ez|    \xi^2 +f^\ez|D\xi||\xi|)\, dx.
\end{align*}
\item[(iv)] If
  $3/2< \alpha<2$, then
\begin{align*}
&  \int_{U}  |D^2u^\ez Du^\epsilon |^2  (|Du^\ez|^2+\kappa)^{   { \alpha-2}}  \xi^2 \, dx
+   \ez\int_{U}  | \Delta u^\epsilon |^2  (|Du^\ez|^2+\kappa)^{{ \alpha-2}} \xi^2 \, dx\\
&\quad\le  C(\alpha) \int_{U}(|Du^{\ez}|^ 2+\kappa)^{ \alpha}(|D\xi|^2+|D^2\xi||\xi|)  \,dx
 +  C(\alpha) \left|\int_{U}        u^\ez_i f^\ez_i  (|Du^{\ez}| ^2+\kappa) ^{  \alpha-2 } \xi^2\, dx\right|\\
  &\quad\quad+   C(\alpha)  \ez \left[\int_{U}     (\Delta  u^\epsilon)^2
     \xi^2  \, dx\right]^{1/2}   \left[\kappa^{ 2\alpha-6}\int_{U}     (f^\ez )^2   \xi^2  \, dx + \kappa^{ 2\alpha-5 }  \int_{U}
    |D^2u^\ez Du^\ez|^2 \,dx\right]^{1/2}   \\
&\quad\quad +C(\alpha)  \kappa^{{ \alpha-3/2}} \int_{U}(
   |D  f^\ez|    \xi^2 +f^\ez|D\xi||\xi|)\, dx.
\end{align*}
\end{enumerate}
\end{lem}

As a consequence of Lemma \ref{DDuez},  we  have the following uniform  Sobolev estimates of $|Du^\ez|^\az$ or $(|Du^\ez|^2+\kappa)^{\alpha/2}$ for all $\alpha>0$.
\begin{lem}\label{uniform DDuez}

(i) If   $  \alpha =2$  or $ \alpha\ge 3$,   we have
$ D|Du^\ez| ^{  \alpha  } \in L^2_\loc(U) $ uniformly in $\ez>0$.

(ii) If $  \alpha \in (0,2)\cup(2,3)$ and   $\kappa>0$,  we have
$D(|Du^\ez|^2+\kappa )^{\alpha/2} \in L^{ 2}_\loc(U)$ uniformly in $\ez>0$.

(iii) If  $  \alpha \in (3/2,2)\cup(2,3)$, then  for all $V\Subset U$ we have
 $$\sup_{\kz\in(0,1)}\limsup_{\ez\to0}\|D(|Du^\ez|^2+\kappa )^{\alpha/2}\|_ {L^{ 2}(V)}<\fz.$$

(iv) If $\alpha\in(0,3/2]$ and $p\in[1,3/(3-\alpha))$, then for  all $V\Subset U$  we have
 $$\sup_{\kz\in(0,1)}\limsup_{\ez\to0}\|D(|Du^\ez|^2+\kappa )^{\alpha/2}\|_{L^{ p}(V)}<\fz.$$

%

%
\end{lem}
\begin{proof}
(i)
Fix arbitrary    ball $B\Subset 2 B\Subset U$ with radius $R$.  For $\alpha>3/2$ and $\kappa\in[0,1]$, by Lemma \ref{uniform Sobolev}  we  have
\begin{align}
M_{ \alpha}(B):=&\sup_{\kappa\in[0,1]}\sup_{\ez\in(0,\ez_U]}\left [\bint_{2B} (|Du^\ez|^2+\kappa)^{ 2\alpha }\,dx+\int_{2B}|Df^\ez|(|Du^\ez|^{2}+\kappa)^{{ \alpha-3/2}}\,dx\right ]\nonumber\\
\le& C(\alpha)\sup_{\ez\in(0,\ez_U]}  \{1+\|
Du^\ez\| ^{ 2\alpha }_{L^{ 2\alpha }(2B)}+ \|Df^\ez\|^q_{L^q(B)}[1+ \|
Du^\ez\| ^{2\alpha-3}_{L^{(2\alpha-3)q/(q-1)}(2B)}] \}\nonumber\\
<&\fz.
\end{align}
By taking suitable cut-off functions $\xi$ in Lemma \ref{DDuez} (i), we have
\begin{align*}
&   \int_B  |D | Du^\epsilon  |^2 |^2   \,dx +\ez\int_B (\Delta u^\ez)^2 \,dx
  \le C  M_1(B) +C  =:\wz M_1(B)  <\fz
\end{align*}

  For $ \alpha\ge 3$, by Lemma \ref{DDuez} (ii) with a suitable  cut-off functions $\xi$   we have
\begin{align*}
&
 \int_{B}  |D | Du^\epsilon|^{\alpha}  |^2     \,dx  +
 {\ez}   \int_B  (\Delta u^\ez)^2 |Du^{\ez}| ^{2\alpha-4}       \,dx \\
&\quad\le C(\alpha) M_{ \alpha}(B) +   C(\alpha)\ez^{1/2} [\wz M_1(B) ]^{1/2}
\|f^\ez\|_{C(2\overline B)} \|Du^\ez\|^{2{ \alpha-1}+4}_{ L^{4{ \alpha-1}-8}(2B)},
\end{align*}
which, by Lemma \ref{uniform Sobolev} and $\|f^\ez\|_{C(2\overline B)} \le 2\|f\|_{C^0(\overline U)}$, is bounded uniformly in $\ez$.

(ii) For $\alpha\in(0,2)$ and $\kz>0$, note that
$$ |D (|Du^\ez|^2+\kz)^{\alpha/2}|= \frac{\alpha }2(|Du^\ez| ^2+\kappa)^{ \alpha/2-1 } |D|Du^\ez|^2|\le \frac{\alpha }2\kappa ^{ \alpha/2-1 } |D|Du^\ez|^2|,
$$
and hence
  $$\|D (|Du^\ez|^2+\kz)^{\alpha/2}\|_{L^2(V)}\le   \kappa^{ \alpha/2-1 }\|D|Du^\ez|^2\|_{L^2(V)}\quad\forall V\Subset U.$$
This together with
   $D|Du^\ez|^2\in L^2_\loc(U)$ uniformly in $\ez>0$  implies that
   $D (|Du^\ez|^2+\kz)^{\alpha/2}\in  L^2_\loc(U)$ uniformly in $\ez>0$.

  For $\alpha\in(2,3)$ and $\kz>0$, note that
\begin{align*}
|D (|Du^\ez|^2+\kz)^{\alpha/2}|&= \frac{\alpha }4(|Du^\ez| ^2+\kappa)^{ \alpha/4-1 } |D(|Du^\ez|^2+\kz)^2|\\
&\le \frac{\alpha }4 \kappa ^{ \alpha/4-1 } [|D |Du^\ez|^4|+  2\kz|D|Du^\ez|^2|].
\end{align*}
Hence,
 $$\|D (|Du^\ez|^2+\kz)^{\alpha/2}\|_{L^2(V)}\le   C(\alpha)[\kappa^{ \alpha/4-1 }\|D|Du^\ez|^4\|_{L^2(V)}+ \kappa^{ \alpha/4+1 }\|D|Du^\ez|^2\|_{L^2(V)}].$$
From this   and
   $D|Du^\ez|^2,\,D|Du^\ez|^4\in L^2_\loc(U)$ uniformly in $\ez>0$, it follows that
   $D (|Du^\ez|^2+\kz)^{\alpha/2}\in  L^2_\loc(U)$ uniformly in $\ez>0$, as desired.

(iii) Let $B\Subset 4B\subset U$. For $2<  \alpha <3$ and $\kappa\in(0,1)$,   by Lemma \ref{DDuez} (ii) with a suitable  cut-off functions $\xi$, we have
\begin{align}\label {ex1-}
&
\int_{B}  |D   (|Du^\ez|^2+\kappa)^{   \alpha}  |^2    \, dx
+   \ez\int_B  | \Delta u^\epsilon |^2  (|Du^\ez|^2+\kappa)^{{ \alpha-2}}  \, dx \nonumber
\\
&\quad\le C(\alpha)M_{ \alpha}(B) +   C(\alpha)\ez^{1/2} \kappa^{ \alpha-3} [\wz M_1( B)]^{1/2}
\|f^\ez\|_{C^0(2\overline B)} \nonumber\\
&\quad\quad +C(\alpha)\kappa^{{ \alpha-3/2}} [ \|D  f^\ez\|_{L^1(2B)}+R\|f^\ez\|_{C^0(2\overline B)}]^{1/2}.
\end{align}
  For $3/2< \alpha <2$ and $\kappa\in(0,1)$, by Lemma \ref{DDuez} (iii) with a suitable  cut-off functions $\xi$, we have
\begin{align}\label{ex2-}
&
\int_{B}  |D   (|Du^\ez|^2+\kappa)^{   \alpha}  |^2 \, dx
+   \ez\int_B  | \Delta u^\epsilon |^2  (|Du^\ez|^2+\kappa)^{{ \alpha-2}}  \, dx\nonumber\\
&\quad\le C(\alpha)M_{ \alpha}(B) +    C(\alpha)\ez^{1/2} (\wz M_1( B) )^{1/2}    [\kappa^{ \alpha-3}
  \|f^\ez\|_{C^0(2\overline B)} + \kappa^{ \alpha-5/2} [\wz M_1( B) ]^{1/2}] \nonumber \\
&\quad\quad +C(\alpha)\kappa^{{ \alpha-3/2}} [ \|D  f^\ez\|_{L^1(2B)}+R\|f^\ez\|_{C^0(2\overline B)}]^{1/2}.
\end{align}

 Since  $\|f^\ez\|_{C^0(2\overline B)} \le \|f\|_{C^0(\overline{\wz  U})}$ and  $$\|D f^\ez\|_{L^1(2B)} \le C(B)\|Df^\ez\|_{L^q(2B)}\le 2C(B)\|Df\|_{L^q(\wz U)},$$
 letting $\ez\to0$ in \eqref{ex1-} and \eqref{ex2-}, by $\alpha>3/2$ we have
\begin{align*}
&
\limsup_{\ez\to0}\int_{B}  |D   (|Du^\ez|^2+\kappa)^{   \alpha}  |^2    \, dx
+   \ez\int_B  | \Delta u^\epsilon |^2  (|Du^\ez|^2+\kappa)^{{ \alpha-2}}  \, dx \nonumber
\\
&\quad\le C(\alpha)M_{ \alpha}(B)
  +C(\alpha)\kappa^{{ \alpha-3/2}} \sup_{\ez\in(0,1)}[ \|D  f^\ez\|_{L^1(2B)}+R\|f^\ez\|_{C^0(2\overline B)}]^{1/2}\nonumber\\
  &\quad\le C(\alpha)M_{ \alpha}(B)
  +C(\alpha)   [ \|Df\|_{L^q(\wz U)}+R\|f \|_{C^0(\overline {\wz U})}]^{1/2}
\end{align*}
as desired.

(iv) For $\alpha\in(0,3/2)$ and $p\in(1,3/(3-\alpha))$, observing $3/2<3/p-3/2+\alpha $  we let $\beta \in (3/2,3/p-3/2+\alpha)$.
For   $\kappa\in(0,1)$,  write
\begin{align*} |D(|Du^\ez|^2+\kappa)^{\alpha/2} |^p &= C(\alpha,\beta)
(|Du^\ez|^2+\kappa)^{p(\alpha-\beta)/2} |D(|Du^\ez|^2+\kappa)^{\beta/2}|^p\nonumber\\
&
\le C(\alpha,\beta,p)[|D|Du^\ez|^\beta|^2+ (|Du^\ez|^2+\kappa)^{ p(\alpha-\beta)/(2-p)}]. \nonumber
\end{align*}
Noting    $\beta<3/p-3/2+\alpha$, that is, $\beta-\alpha<3/p-3/2=3(2-p)/2p $ implies that
$ p(\alpha-\beta)/(2-p) >-3/2$. Write
$ p(\alpha-\beta)/(2-p) =\gamma-3$, we know that $\gamma>3/2$.
If $\beta\ne 2$ sufficiently close to $3/p-3/2+\alpha$, we actually have $\gamma<2$.
Observe that \begin{align*}
(f^\ez)^2(|Du^\ez|^2+\kappa)^{ \gamma-3} &=(|Du^\ez|^2+\kappa)^{\gamma-3} (\Delta_\fz u^\ez+\ez\Delta u^\ez)^2  \\
&\le 2
(|Du^\ez|^2+\kappa)^{ \gamma-2}  | D(|Du^\ez|^2+\kappa)|^2   +2\ez^2(|Du^\ez|^2+\kappa)^{\alpha-3} (\Delta u^\ez )^2 \\
&\le C
 | D (|Du^\ez|^2+\kappa)^{ \gamma/2 } |^2 +C\ez^2  \kappa ^{\alpha-3} (\Delta u^\ez )^2.
\end{align*}

We have
\begin{align*}
&\liminf _{\ez\to0}\int_{B}   |D   (|Du^\ez |^2+\kappa)^{   \alpha}  |^p   \, dx \nonumber\\
&\le C(\alpha,\beta,p) \liminf _{\ez\to0}\int_{B}   |D(|Du^\ez|^2+\kappa)^{\beta/2}|^2\,dx \\
&\quad +  C(\alpha,\beta,p)
\limsup _{\ez\to0} \|1/f\|^{-2}_{C^0(\overline B)} \int_{B}  | D (|Du^\ez|^2+\kappa)^{ \gamma/2 } |^2\,dx\\
&\quad+  C(\alpha,\beta,p) \ez  \kappa ^{\alpha-3}\|1/f\|^{-2}_{C^0(\overline B)} \limsup _{\ez\to0}\ez\int_{B}   (\Delta u^\ez )^2\,dx.
\end{align*}

Note that $\ez  \int_B (\Delta u^\ez )^2\,dx\le \wz M_1(B)$ as given in the proof of Lemma \ref{uniform DDuez} (i),
and that $f^\ez$ is  bounded away from $0$ on $B$ uniformly in $\ez\in(0,\ez_U)$.
We have
\begin{align*}
&\liminf _{\ez\to0}\int_{B}   |D   (|Du^\ez |^2+\kappa)^{   \alpha}  |^p   \, dx \nonumber\\
&\le C(\alpha,\beta,p) \liminf _{\ez\to0}\int_{B}   |D(|Du^\ez|^2+\kappa)^{\beta/2}|^2\,dx \\
&\quad +  C(\alpha,\beta,p,f,B)
\limsup _{\ez\to0}   \int_{B}  | D (|Du^\ez|^2+\kappa)^{ \gamma/2 } |^2\,dx,
\end{align*}
by  Lemma \ref{uniform DDuez} (iii) for $\bz,\gz\in(3/2,2)$,     which   is uniform bounded in $\kz\in(0,1)$ as desired.

This completes the proof of Lemma \ref{strong converge}.
\end{proof}

\begin{rem}
Note that   the proof of Lemma \ref{uniform DDuez} (ii) uses $\kappa>0$, and hence, for any $\alpha\in(0,2)\cup(2,3)$, does not give
  $D|Du^\ez|^\alpha \in L^{ 2}_\loc(U)$ uniformly in $\ez>0$.

For $\alpha\in(2,3)$, if   $ |Du^\ez|  \in L^\fz_\loc(U)$ is unavailable (for example, under $f\in W^{1,\fz}_\loc(\Omega))$,
by $$|D |Du^\ez|^\alpha|= \frac{\alpha }2 |Du^\ez|  ^{ \alpha-2 } |D |Du^\ez|^2  |,
$$
one would conclude $|D |Du^\ez|^\alpha|\in L^2_\loc(U)$ uniformly in $\ez>0$ from
$|D |Du^\ez|^2  |  \in L^2_\loc(U)$ uniformly in $\ez>0$.
But in general,   $ |Du^\ez| \in L^\fz_\loc$ uniformly in $\ez>0$ is unavailable as Remark \ref{sobolevbv} for details.

\end{rem}

Taking $\phi=|Du^\ez|^4[(u^\ez-P) \xi]^2$ in \eqref{functional}
we obtain the following flatness. The details are postponed to Section 6.
\begin{lem}\label{flat}
For any  linear function $P$, we have
\begin{align*}
& \int_{U}\langle Du^{\ez}, D u^\ez -DP \rangle ^2|Du^\ez|^6  \,dx\\
& \le C\left[ \int_{U}   |D^2u^\ez Du^\ez|^2   \xi^2 \,dx \right]^{1/2}
\left[ \int_{U}    |Du^{\ez}|  ^{12}       |Du-DP|^2    (u^\ez-P)^2\xi^2 \,dx\right]^{1/2}\\
&\quad+C  \int_{U}    [|Du^{\ez}|  ^8(|D\xi|^2 +|D^2\xi||\xi|)+|Du^{\ez}|  ^2 (f^\ez)^2\xi^2 ](u^\ez-P)^2  \,dx\\
&\quad +C  \left|\int_{U}
f^\ez_iu^\ez_i|Du^{\ez}|  ^4      (u^\ez-P)^2\xi^2  \,dx\right|.
 \end{align*}

\end{lem}

 \section{Proofs of Theorem \ref{mainthm} and a flatness when $ f\in (\cup_{q>1}W^{1,q}_\loc(\Omega))\cap C^0(\Omega)$ }
Suppose $\Omega \Subset\mathbb R^2$, and $ f\in W^{1,q}_\loc(\Omega)\cap C^0(\Omega)$ for some $q>1$ with $f>0$ in $\Omega$. Let $u\in C^0(\Omega)$ be a viscosity solution to  $-\Delta_\fz u=f$ in $\Omega$.
Given arbitrary domain $U \Subset\mathbb R^2$,  let $\ez_U$ and $\{f^\ez\}_{\ez\in(0,\ez_U]}$ as in Section 3.
For each  $  \ez \in(0,\ez_U]$,  let $ u^\epsilon\in C^\infty(U)\cap C(\overline U)$ be a   solution to  \eqref{infty ezf}.

The following convergence follows from Lemma \ref{uniform Sobolev}, Lemma \ref{maxandboundary}  and Lemma \ref{unique}.
\begin{lem}\label{uniform conv}
$u^\ez\to u$  in $C^0( U)$ as $\ez\to0$.
\end{lem}

\begin{proof}
By Lemma \ref{uniform Sobolev}  we know that for any $\gamma\in(0,1)$, $u^\ez\in C^{0,\gamma} (U)$ uniformly in $\ez\in(0,\ez_U]$.
Thus, there exists a function $\hat u\in C^{0,\gamma} (U)$ such that, up to a subsequence,
$
u^\epsilon\rightarrow \hat u \ {\rm{in}}\ C^0(U).
$ as $\ez\to0$.
By Lemma \ref{maxandboundary},
 for sufficient small   $ \epsilon>0$  we have
\begin{equation}\label{bdry_uniform}
\big|u^\epsilon(x)-u (x_0)\big|\le C|x-x_0|^\gamma, \ \forall\ x\in U, \ x_0\in\partial U.
\end{equation}
Note that
   $u^\ez(x)\to \hat u(x)$ for $x\in U$ as $\ez\to0$. Letting $\ez\to0$ in \eqref{bdry_uniform}, we obtain
$$|\hat{u}(x)-u(x_0)|\le C|x-x_0|^\gamma, \ \forall\ x\in U, \ x_0\in\partial U.$$
Thus $\hat{u}\in C(\overline U)$ with $\hat u\equiv u $ on $\partial U$.
By the compactness property of viscosity solutions of elliptic equations (see Crandall-Ishii-Lions \cite{CIL}), we know that $\hat{u}
\in C(\overline U)$ is a viscosity solution to the equation $\Delta_\fz v=-f$ in $U$.
 Since $\hat{u}\equiv u$ on $\partial U$ and $f> 0$ in $\overline U$, it follows from Lemma \ref{unique} that $\hat{u}=u$ in $U$. This also implies that
$u^\epsilon\rightarrow u$ in $C^0 (U)$ as $\epsilon\rightarrow 0$.\end{proof}

\begin{rem}\label{rq}
\begin{enumerate}
\item[(i)] When $0<f\in  W^{1,\fz}_\loc(\Omega)$, it is already proved in \cite{l,lmz} that
$u^\ez \to u$  in $C^0_\loc ( U)$ as $\ez\to0$.
 Note that the assumption  $0<f\in  (\cup_{q>1}W^{1,q}_\loc(\Omega))\cap C^0(\Omega)$   used here  is much weaker than
$0<f\in  W^{1,\fz}_\loc(\Omega)$.

\item[(ii)] Under the assumption  $0<f\in  W^{1,1}_\loc(\Omega)\cap C^0(\Omega)$ or $0<f\in BV_\loc(\Omega)\cap C^0(\Omega)$,
it is still unknown whether  $u^\ez \to u$  in $C^0_\loc ( U)$  as $\ez\to0$ or not.
Note that by our above argument, the convergence $u^\ez \to u$  in $C^0_\loc ( U)$ as $\ez\to0$ would follow from
 $u^\ez\in W^{1,p}_\loc(U)$ uniformly in $\ez>0$ for some $p\in(2,\fz]$,
but which, as indicated by Remark \ref{sobolevbv}, is   available only  when  $0<f\in  (\cup_{q>1}W^{1,q}_\loc(\Omega))\cap C^0(\Omega)$.
\end{enumerate}
\end{rem}

Lemma \ref{uniform DDuez}, Lemma \ref{DDuez} and Lemma \ref{flat} allow us to prove the following Sobolev convergence, which is crucial to prove Theorem \ref{mainthm} under $0<f\in W^{1,q}_\loc (\Omega)\cap C^0(\Omega)$ with $q>1$.
\begin{lem} \label{strong converge}

\begin{enumerate}
\item[(i)] If   $  \alpha \ge3$ or $  \alpha =2$,  we have
 $|Du^\epsilon|^{\alpha}\to |Du|^{\alpha}$  in $ L^p_{\loc}(U)$
 for all $p\in[1,\infty)$ and weakly in $W^{1,2}_\loc(U)$ as $\ez\to0$.

 Moreover,  $u^{\ez} \to u$   in $W^{1,\,p}_{\loc}(U)$ for all $p\in[1,\infty)$ as $\ez\to0$.

\item[(ii)] If $ \alpha\in(0,2)\cup(2,3)$,  for each $\kappa\in(0,1]$ we have
 $(|Du^\epsilon|^2+\kappa)^{\alpha/2}\to (|Du |^2+\kappa)^{\alpha/2}$  in $ L^p_{\loc}(U)$ for all $p\in[1,\infty)$ and weakly in $W^{1,2}_\loc(U)$
 as $\ez\to0$.

\item[(iii)]  If $ \alpha\in(3/2,2)\cup(2,3)$, we have $(|Du |^2+\kappa)^{\alpha/2}\to  |Du | ^{  \alpha  }$  in $ L^p_{\loc}(U)$ for all $p\in[1,\infty)$ and weakly in $W^{1,2}_\loc(U)$ as $\kz\to0$ .

\item[(iii)]  If $ \alpha\in(0,3/2]$ and $p\in[1, 3/(3-\alpha) )$,   we have $(|Du |^2+\kappa)^{\alpha/2}\to  |Du | ^{  \alpha  }$  in $ L^t_{\loc}(U)$ for all $t\ge1$ and  weakly in $W^{1,p}_\loc(U)$  as $\kz\to0$.
 \end{enumerate}
\end{lem}

\begin{proof}

{\it Proof of (i)}  By Lemma \ref{uniform DDuez}  for $\alpha=2$  we know that
$D|Du^\epsilon| ^2\in W^{1,2}_\loc(U) $ uniformly in $\ez>0$.
From the weak compactness of $W_\loc^{1,2}(U)$, it follows that
$|Du^\epsilon| ^{2} $ converges as $\ez\to0$ (up to some subsequence) to some function $h  $ in $L^p_\loc(U)$
and weakly  in $W^{1,2}_\loc(U)$.

It suffices to prove  $h =|Du|^2 $   almost everywhere.
Indeed, assume this holds for the moment. We then have
 $|Du^\epsilon| ^{2}\to  |Du|^2$ in $L^p_\loc(U)$ for all $p\ge1$ as $\ez\to0$.
This together with $Du^\ez\to Du$ weakly in $L^p_\loc(U)$ implies that  $Du^\ez\to Du$   in $L^p_\loc(U)$.
For $\alpha\ge3$,
$$ |D u^\ez| ^{  \alpha  }=(|D u^\ez| ^2)^{  \alpha  /2}\to (h  )^{\alpha/2}=(|Du |^2 )^{\alpha/2}=
 |Du | ^{  \alpha  }$$
almost everywhere as $\ez\to0$. By Lemma \ref{uniform DDuez}  for   $\alpha\ge3$ we also know that
$D|Du^\epsilon| ^{\alpha}\in W^{1,2}_\loc(U) $ uniformly in $\ez>0$.
From the weak compactness of $W_\loc^{1,2}(U)$ again and
$ |D u^\ez| ^{  \alpha  }\to
 |Du | ^{  \alpha  }$, it follows that
$|Du^\epsilon| ^{\alpha} \to |Du | ^{  \alpha  }$   in $L^p_\loc(U)$ for any $p\in[1,\fz)$ and weakly in $L^2_\loc(U)$ as $\ez\to0$.

To   prove  $h =|Du|^2 $ almost everywhere, we only need to prove $h (\bar x)=|Du(\bar x)|^2 $ for all
$\bar x\in U$ such that
  $u$ is differentiable at $\bar x$, and   $\bar x$ is the Lebesgue point of $ Du $
and  $[h]^N$   with $N\ge  5q/(q-1)$.
Note that the set of such $\bar x$ has full measure in $U$.

If $h(\bar x)=0$, then by $u^\ez\to u$ in $C^0_\loc(U)$  (see Lemma \ref{uniform Sobolev}) and
$|Du^\epsilon| ^{2}\to h $ in $L^2_\loc(U)$ as $\ez\to0$, we have
\begin{align*}|Du(\bar x)| &\le C\limsup_{R\to0} \frac1R\bint_{B(\bar x, R)} |u- u_{B(\bar x, R)}|\\
& =  C
\limsup_{R\to0} \limsup_{\ez\to0}\frac1R\bint_{B(\bar x, R)} | u^\ez-   u^\ez_{B(\bar x, R)}|\,dx \\
&\le C\limsup_{R\to0} \limsup_{\ez\to0} \left [\bint_{B(\bar x, R)} |D   u^\ez| ^2\,dx\right ]^{1/2}\\
&\le C\limsup_{R\to0}   \left [\bint_{B(\bar x, R)}  h \,dx\right ]^{1/2}\\
&\le C[h(\bar x)]^{1/2}\\
&=0
\end{align*}
as desired.

  Below, assume
 $h(\bar x)>0$.
Note that for $p=1,\cdots, N$ , we have $$\lim_{r\to0}\bint_{B(\bar x,r)}h^p\,dx = [h(x)]^p,$$
which implies that there exists $r_{\bar x}<\dist(\bar x, \partial U)/8$ such that for all $r<r_{\bar x}$, we have
$$ \bint_{B(\bar x,r)}h^p\,dx \le  2[h(x) ]^p.$$
Considering $|Du^\ez|^2\to h$ in $L^{p}_\loc(U)$ as $\ez\to0$, we know that  for each $r\in(0,r_{\bar x})$,
 there exists   $ \ez_{\bar x,r} \in(0,r)\cap(0,\ez_U)$ such that for  any $\ez \in(0,\ez_r)$,
  $$ \bint_{B(\bar x,r)}|Du^\ez|^{2p}\,dx \le 4 [h(x) ]^p.$$

Moreover,  for any $\lambda\in(0,1)$, thanks to the differentiability at $\bar x$ of $u$,
there exists  $r_{\lambda,\bar x}\in (0, r_{\bar x})$ such that for any $ r\in(0,r_{\lambda,\bar x})$, we have
$$ \sup_{B(\bar x, 2r)}\frac{|u(x) - u(\bar x)-\langle Du(\bar x),(x-\bar x)\rangle | }{r }\le \lambda.$$
By Lemma \ref{uniform Sobolev}, for arbitrary     $ r\in(0,r_{\lambda,\bar x})$, there exists $\ez_{\lambda,\bar x, r}\in(0, \ez_{\bar x,r}]$ such that for all $\ez\in(0,\ez_{\lambda,\bar x,r })$, we have
$$ \sup_{B(\bar x, 2r)}\frac{|  u^\ez(x) -    u^\ez(\bar x)-\langle Du(\bar x),(x-\bar x)\rangle| }{r }\le 2\lambda.$$

Let $P(x)= u^\ez(\bar x)-\langle Du(\bar x),(x-\bar x)\rangle$ in Lemma \ref{flat}. Then $DP=Du(\bar x)$ in $U$.
For all balls $B= B(\bar x, r)$ with $r\in(0,r_{\lambda,\bar x})$, choose a suitable  cut-off function $\xi$  associated to $B$.
For  $ \ez\in(0,\epsilon_{\lambda,\bar x,r})$,
from $|u^\ez-P|\le 2r\lz$ on $2B$ and  Lemma \ref{flat} it follows that
\begin{align*}
& \int_{B}\langle Du^{\ez}, D u^\ez -Du(\bar x) \rangle ^2|Du^\ez|^6  \,dx\\
& \quad\le r^2\lz \left[ \int_{2B}   |D^2u^\ez Du^\ez|^2     \,dx \right]^{1/2}
\left [ \bint_{2B}    |Du^{\ez}|  ^{12}       |Du^\ez-Du(\bar x)|^2      \,dx\right]^{1/2}\\
&\quad\quad+C  r^2\lz^2 \int_{2B}    [r^{-2}|Du^{\ez}|  ^8  +|Du^{\ez}|  ^2 (f^\ez)^2
+|Du^{\ez}|  ^5  | D f^\ez|      ]    \,dx
 \end{align*}
Note that  by Lemma \ref{DDuez}, we have
\begin{align*}
&   \int_B  |D | Du^\epsilon  |^2 |^2   \,dx\bint_{2B}    |Du^{\ez}|  ^{12}       |Du^\ez-Du(\bar x)|^2      \,dx  \\
&\quad\le
  C \left[\bint_{2B} |Du^\ez|^4\,dx +\int_{2B}|Df^\ez| \,dx+1 \right] \left[\bint_{2B}    |Du^{\ez}|  ^{14}         \,dx + |Du(\bar x)|^2\bint_{2B}    |Du^{\ez}|  ^{12}         \,dx \right]\\
  &\quad\le  C\{[h(\bar x)]^2 +1+ (r_{ \bar x})^{1-1/q} \|Df\|_{L^q(2B)}\} [h(\bar x)]^6[h(\bar x)+ |Du(\bar x)|^2],
\end{align*}
where we use$$ \int_{2B}|Df^\ez| \,dx\le r^{1-1/q} \|Df^\ez\|_{L^q(2B)} \le 2 (r_{ \bar x})^{1-1/q} \|Df\|_{L^q(4B)}.$$
  Moreover,
\begin{align*}&\int_{2B}    [r^{-2}|Du^{\ez}|  ^8  +|Du^{\ez}|  ^2 (f^\ez)^2
+|Du^{\ez}|  ^5  | D f^\ez|      ]    \,dx\\
&\quad\le C\bint_{2B}|Du^{\ez}|  ^8 \,dx + r^2\|f^\ez\|^2_{C^0(\overline U)} \bint_{2B}|Du^{\ez}|  ^2 \,dx
\\
&\quad\quad+r^{2-2/q}\|Df^\ez\|_{L^q(U)}\left[\bint_{2B} |Du^\ez|^{5q/(q-1)}\,dx\right]^{(q-1)/q}\\
&\quad\le  C
[h(\bar x) ]^4+C r^2\|f \|^2_{C^0(\overline U)} h(\bar x)   +
C( r_{\bar x})^{2-2/q}\|Df \|_{L^q(U)} [h(\bar x) ]^{N/5}
\end{align*}
 We therefore conclude that
\begin{align}\label{flatuez}
 \bint_{B(\bar x,r)} (|Du ^{\ez}|^2- \langle Du(\bar x),D u^\ez\rangle  )^2|Du ^{\ez}|^6 \,dx\le
 C (\bar x,  \|f\|_{C(\bar U)},\|Df\|_{L^q(U)}) \lambda .
\end{align}

Since
$ |D u^\ez|^2 \to h$ in $L^2_\loc(U)$ and $D u^\ez \rightharpoonup Du$ weakly in $L^2_\loc(U)$ as $\ez\to0$,
applying \eqref{flatuez} with $r\in(0,r_{\lambda,\bar x})$ we have
\begin{align*}
\bint_{B(\bar x,r)} (h- \langle Du(\bar x),Du \rangle  )^2 |h|^3\,dx
 &\le   \liminf_{\ez\to0}\bint_{B(\bar x,r)} (|Du ^{\ez}|^2- \langle Du(\bar x),D  u^\ez\rangle  )^2|Du ^{\ez}|^6 \,dx\\
 &
  \le
C   (\bar x,\|f\|_{C^0(\overline U)},\|Df\|_{L^q(U)})\lambda   \quad
\end{align*}
Since $\bar x$ is a Lebesgue point of $[h]^4$ with $p=1,\cdots,N$ and $ Du $, via  H\"older's inequality, we obtain
\begin{align*}|h(\bar x)-|Du( \bar x)|^2|^2|h(\bar x)|^3&=\lim_{r\to 0} \bint_{B(\bar x,r)} |h- \langle Du(\bar x),Du \rangle  ||h|^3\,dx\\
&\le   C (\bar x,  \|f\|_{C(\bar U)},\|Df\|_{L^q(U)})\lambda  . \end{align*}
By $h(\bar x)>0$, letting $\lambda\to0$ we have $h(\bar x)=|Du(\bar x) |^2$ as desired.

{\it Proof of (ii)} For $ \alpha\in(0,2)\cup(2,3)$ and $\kappa\in(0,1)$, by Lemma \ref{uniform DDuez} (ii) we know that $ (|Du^\epsilon| ^2+\kappa)^{\alpha/2} \in W^{1,2}_\loc(U)$ uniformly in $\ez>0$.
By (i) $ (|Du^\epsilon| ^2+\kappa)^{\alpha/2} \to  (|Du | ^2+\kappa)^{\alpha/2} $ almost everywhere.
From this and the weak compactness of $W_\loc^{1,2}(U)$, it follows that
$(|Du^\epsilon| ^2+\kappa)^{\alpha/2}\to  (|Du | ^2+\kappa)^{\alpha/2} $
in $L^p_\loc(U)$ for all $p\ge1$
and weakly  in $W^{1,2}_\loc(U)$ as $\kappa\to0$.

{\it Proof of (iii)} For $\alpha\in(3/2,2)\cup(2,3)$,  by (ii) and Lemma \ref{uniform DDuez} (iii) we have $(|Du | ^2+\kappa)^{\alpha/2} $       in  $W^{1,2}_\loc(U)$ uniformly in $\kappa\in(0,1)$.
Observing that $(|Du | ^2+\kappa)^{\alpha/2}  \to |Du|^{\alpha}$ almost everywhere as $\kz\to0$,  by Lemma \ref{lipu} and Lebesgue's theorem, we have $(|Du | ^2+\kappa)^{\alpha/2}  \to |Du|^{\alpha}$  in $L^t_\loc(U)$ for all $t\ge1$ as $\kz\to0$.
By the compactness of $W^{1,2}_\loc(U)$ again,    we  have
$|Du|^{\alpha}\in W^{1,2}_\loc(U)$, $(|Du | ^2+\kappa)^{\alpha/2}  \to |Du|^{\alpha}$   weakly in $W^{1,2}_\loc(U)$ as $\kappa\to0$.

{\it Proof of (iv)} For $\alpha\in(0,3/2]$ and $p\in[1,3/(3-\alpha))$,
by (ii) and Lemma \ref{uniform DDuez} (iv)  we have $(|Du | ^2+\kappa)^{\alpha/2} \in W^{1,p}_\loc(U)$ uniformly in $\kappa\in(0,1)$.
  Similarly to (iii), observing that $(|Du | ^2+\kappa)^{\alpha/2}  \to |Du|^{\alpha}$ almost everywhere as $\kz\to0$,  by Lemma \ref{lipu} and Lebesgue's theorem, we have
 $(|Du | ^2+\kappa)^{\alpha/2}  \to |Du|^{\alpha}$  in $L^t_\loc(U)$ for all $t\ge1$ as $\kz\to0$.
 By the compactness of $W^{1,p}_\loc(U)$ again,    we  have
 $|Du|^{\alpha}\in W^{1,p}_\loc(U)$, $(|Du | ^2+\kappa)^{\alpha/2}  \to |Du|^{\alpha}$   weakly in $W^{1,p}_\loc(U)$ as $\kappa\to0$.

This completes the proof of Lemma \ref{strong converge}.
\end{proof}

\begin{proof}[Proofs of Theorem \ref{mainthm} when $ 0<f\in (\cup_{q>1}W^{1,q}_\loc(\Omega))\cap C^0(\Omega)$.]
By Lemma \ref{strong converge}, we have  that
$|Du|^{\alpha}\in W^{1,2}_\loc(U)$ when $\alpha>3/2$, and $|Du|^{\alpha}\in W^{1,p}_\loc(U)$ when $\alpha\in(0,3/2]$ and $p\in[1,3/(3-\alpha))$.  This gives (ii), and also reduces
(i) to verifying \eqref{strongest}, where  we note that \eqref{ex01} follows from \eqref{strongest} and Lemma \ref{lipu}.
Moreover, note that (iii) follows from (i) and (iv) as indicated by Remark \ref{relthm} (i).
 So below we only need to prove   \eqref{strongest},  and (iv), that is, \eqref{chain} and \eqref{pw}.

{\it Proof of \eqref{strongest}.}
Let $U\Subset \Omega$ such that $V:={\rm supp}\, \xi\Subset U$.
 For $\ez\in(0,1)$, let $u^\ez$ be a smooth solution to \eqref{infty ezf}.
Note that $u^\ez$ to $u$ in $W^{1,p}_\loc (U)$ for all $p\ge1$ as given in Lemma \ref{strong converge}, and $f^\ez \to f $ in $W^{1,q}(U)$,
For $  \alpha =2$ or $  \alpha \ge3$
we have  $$\int_{U}          f^\ez  _i  u^\ez_i |Du^\ez | ^{2\alpha-4} \xi^2\, dx\to \int_{U}  f _i  u_i |Du | ^{2\alpha-4} \xi^2\, dx,$$
For $\alpha\in(3/2,2)\cup(2,3)$ and $\kappa\in(0,1)$,
we have  $$\int_{U}          f^\ez  _i  u^\ez_i (|Du^\ez | ^2+\kappa)^{ { \alpha-2}} \xi^2\, dx\to \int_{U}  f _i  u_i (|Du | ^2+\kappa)^{ { \alpha-2}} \xi^2\, dx,$$

Letting $\ez\to 0$ in Lemma \ref{DDuez} (i) we have
\begin{align}\label{ex3}
  \int_{U}  |D |   Du|^2 |^2 \xi^2 \,dx&\le \liminf_{\ez\to 0} \left[ \int_{U}  |D |   Du^\ez|^2|^2 \xi^2 \,dx+
   \ez \int_{U}   |   \Delta u^\ez|^2 \xi^2 \,dx\right]\nonumber\\
  &\le C   \liminf_{\ez\to 0}  \int_{U}|Du^\ez | ^{6}(|D\xi|^2+|D^2\xi||\xi|) \,dx
 +  C\liminf_{\ez\to 0} \left| \int_{U}          f^\ez _i  u^\ez_i  \xi^2\, dx\right|\nonumber\\
 &\le C       \int_{U}|Du  | ^6(|D\xi|^2+|D^2\xi||\xi|) \,dx
 +  C  \left|\int_{U}          f  _i  u _i   \xi^2\, dx\right|
\end{align}
as desired.

For $  \alpha \ge3$ letting $\ez\to 0$ in Lemma \ref{DDuez} (ii)  we have
 \begin{align*}
  &\int_{U}  |D |   Du|^{\alpha} |^2 \xi^2 \,dx\\
  &\quad\le \liminf_{\ez\to 0}  \int_{U}  |D |   Du^\ez|^{\alpha} |^2 \xi^2 \,dx\\
  &\quad\le C   \liminf_{\ez\to 0}  \int_{U}|Du^\ez | ^{ 2\alpha }(|D\xi|^2+|D^2\xi||\xi|) \,dx
 +  C\liminf_{\ez\to 0} \left| \int_{U}          f^\ez _i  u^\ez_i |Du | ^{2\alpha-4} \xi^2\, dx\right|\\
 &\quad\quad + C(\alpha)\liminf_{\ez\to 0} \ez\left [\int_{U}     (\Delta  u^\epsilon)^2       \xi^2  \, dx\right]^{1/2}   \left[\int_{U}     (f^\ez )^2 |Du^\epsilon |^{4  \alpha -12}      \xi^2  \, dx\right]^{1/2} \\
 &\quad\le C       \int_{U}|Du  | ^{ 2\alpha }(|D\xi|^2+|D^2\xi||\xi|) \,dx
 +  C  \left|\int_{U}          f  _i  u _i |Du | ^{2\alpha-4} \xi^2\, dx\right|,
\end{align*}
where by \eqref{ex3} and  Lemma \ref{uniform Sobolev} we have
 \begin{align*}&\liminf_{\ez\to 0} \ez\left [\int_{U}     (\Delta  u^\epsilon)^2       \xi^2  \, dx\right]^{1/2}   \left[\int_{U}     (f^\ez )^2 |Du^\epsilon |^{4{ \alpha-1}-8}      \xi^2  \, dx\right]^{1/2} \\
&\quad\le \liminf_{\ez\to 0} \ez^{1/2}\cdot \ez \int_{U}     (\Delta  u^\epsilon)^2       \xi^2  \, dx
   +\liminf_{\ez\to 0} \ez^{1/2} \int_{U}     (f^\ez )^2 |Du^\epsilon |^{4{ \alpha-1}-8}      \xi^2  \, dx \\
   &\quad=0.
      \end{align*}

Similarly,  for $  \alpha \in(3/2,2)\cup(2, 3)$ and $\kappa\in(0,1)$, letting $\ez\to 0$ in Lemma \ref{DDuez} (iii) and (iv),
 we have
 \begin{align*}
  &\int_{U}  |D (|   Du|^2+\kappa)^{\alpha/2} |^2 \xi^2 \,dx\\
  &\quad\le
  \liminf_{\ez\to 0}  \int_{U}  |D (|   Du^\ez|^2+\kappa)^{\alpha/2} |^2 \xi^2 \,dx\\
  &\quad\le C ({ \alpha})     \int_{U}(|Du  | ^2+\kz)^{ \alpha}(|D\xi|^2+|D^2\xi||\xi|) \,dx
 +  C(\alpha) \left| \int_{U}          f _i  u_i( |Du | ^2+\kz)^{ { \alpha-2}} \xi^2\, dx\right|\\
   &\quad\quad +C(\alpha)\kappa^{{ \alpha-3/2}} \int_{U}(  |D  f |    \xi^2 +f |D\xi||\xi|)\, dx.
\end{align*}
Sending $\kappa\to0$, by Lemma \ref{uniform Sobolev} we obtain
 \begin{align}
  \int_{U}  |D  |   Du|  ^{  \alpha  } |^2 \xi^2 \,dx
  &\le
  \liminf_{\kappa\to 0}  \int_{U}  |D (|   Du |^2+\kappa)^{\alpha/2} |^2 \xi^2 \,dx\nonumber\\
  &\le C ({ \alpha })     \int_{U}|Du  | ^{  2\alpha }(|D\xi|^2+|D^2\xi||\xi|) \,dx
 +  C(\alpha) \left| \int_{U}          f _i  u_i  |Du |  ^{ 2\alpha-4} \xi^2\, dx\right|\nonumber
\end{align}
as desired.

{\it Proof of \eqref{chain}.}
We only consider the case
 $\alpha,\tau\in(0,2)\cup (2,3)$; the other cases are similar and easier.
    By Lemma \ref{strong converge} we have, for any $\Phi\in C^\fz_c(U,\mathbb R^2)$, \begin{align*}
\int_U |Du|^\tau \langle D|Du|^\alpha, \Phi\rangle \,dx&=\lim_{\kappa\to0}
\int_U (|Du|^2+\kappa)^{\tau/2} \langle D(|Du|^2+\kappa)^{\alpha/2}, \Phi\rangle\,dx\\
&=\lim_{\kappa\to0}\lim_{\ez\to0}
\int_U (|Du^\ez|^2+\kappa)^{\tau/2}  \langle D(|Du^\ez|^2+\kappa)^{\alpha/2}, \Phi\rangle\,dx\\
&=\frac{\alpha}{\alpha+\tau}\lim_{\kappa\to0}\lim_{\ez\to0}
\int_U \langle D(|Du^\ez|^2+\kappa)^{(\alpha+\tau)/2}, \Phi\rangle\,dx\\
&=\frac{\alpha}{\alpha+\tau} \lim_{\kz\to0}
\int_U  \langle D(|Du |^2+\kappa)^{(\alpha+\tau)/2} , \Phi\rangle\,dx\\
&=\frac{\alpha}{\alpha+\tau}
\int_U  D|Du | ^{ \alpha+\tau }, \Phi\rangle\,dx,
\end{align*}
which implies \eqref{chain} as desired.

{\it Proof of \eqref{pw}.}
By \eqref{chain}, to obtain \eqref{pw} it suffices to prove $(|Du|^2)_iu_i=2f$ almost everywhere. Indeed, assume this holds for the moment.
When $\alpha> 2$, we have $$(|Du|^\alpha)_iu_i= \frac{\alpha}{2} |Du|^{ \alpha-2} (|Du|^2)_iu_i =\alpha |Du|^{\alpha-2}f $$
almost everywhere.
When $\alpha\in(0,2)$, noting that
 $f>0$ implies that $Du$ and $D|Du|^2 $ vanishes only on a set with measure $0$; otherwise
  $(|Du|^2)_iu_i=0\ne 2f$  on a set with positive measure. Therefore
  $D(|Du|^\alpha)  = \frac{\alpha}{2} |Du|^{ \alpha-2}  D|Du|^2$  and hence,  we obtain
  $(|Du|^\alpha)_iu_i =\alpha |Du|^{\alpha-2}f $ similarly.

Finally, we prove $(|Du|^2)_iu_i=2f$ almost everywhere.
By  $Du^\epsilon\to Du $ in $W^{1,2}_\loc(U)$ and     $D|Du^\epsilon|^2 \to D|Du |^2$ weakly in $L^2_\loc(U)$ as given in Lemma~\ref{strong converge},
 we have
\begin{align*}\int_U \langle D|Du|^2,Du\rangle \phi\,dx&=
\lim_{\epsilon\to0} \int_U \langle D|Du^\epsilon|^2,Du^\epsilon\rangle \phi\,dx=
\lim_{\epsilon\to0}   \int_U  2    \Delta_\infty u^\ez \phi\,dx \quad\forall\phi\in C_c^\infty(U).
\end{align*}
Applying $\Delta_\infty u^\ez =-\ez\Delta u^\ez-f^\ez$, we have
\begin{align*}\int_U \langle D|Du|^2,Du\rangle \phi\,dx
&= \lim_{\epsilon\to0} \left[-
2\ez \int_U\Delta  u^\epsilon  \phi\,dx-  \int_U  2   f^\ez \phi\,dx\right]  \quad\forall\phi\in C_c^\infty(U).
\end{align*}
By \eqref{ex3},
\begin{align*}
 \lim_{\epsilon\to0}
 \ez \int_U  \Delta  u^\epsilon  \phi\,dx
  \le \liminf_{\epsilon\to0} \ez \left[\int_U (\Delta  u^\epsilon )^2 |\phi|^2\,dx\right]^{1/2}\,|\rm supp \, \phi|^{1/2}
 =0.
 \end{align*}
   Therefore
$$\int_U \langle D|Du|^2,Du\rangle \phi\,dx= -  \int_U   2   f  \phi\,dx\quad\forall\phi\in C_c (U) $$
as desired.
\end{proof}

Finally, with the aid of Lemma \ref{flat}, Lemma \ref{strong converge} and Lemma \ref{DDuez},
we  prove the following integral flatness for $u$, which we use to prove Theorem \ref{mainthm} when $0<f\in  BV_\loc(\Omega)\cap C^0(\Omega)$.
\begin{lem}\label{flatu}
For any  linear function $P$, we have
\begin{align*}
& \int_{\Omega}\langle Du  , D u  -DP \rangle ^2|Du |^6  \xi^2\,dx\\
& \le  C\left [\int_{\Omega}|Du | ^4 (|D\xi|^2 +|D^2\xi||\xi|) \,dx + \left|\int_{\Omega}   u_if_i \xi^2  \, dx\right| \right ]^{1/2}\\
&\quad\quad\times
\left[ \int_{\Omega}    |Du |  ^{12}       |Du-DP|^2    (u -P)^2\xi^2 \,dx\right]^{1/2}\\
&\quad+C  \int_{\Omega}    [|Du |  ^8(|D\xi|^2 +|D^2\xi||\xi|) +|Du |  ^2 (f )^2\xi^2 ] (u-P)^2  \,dx
\\
&\quad+C \left| \int_{\Omega}    f_iu_i|Du |  ^4\xi^2   (u-P)^2  \,dx\right|.
 \end{align*}
\end{lem}

\begin{proof}
Let $\xi\in C_c^2(\Omega)$, and $U\Subset\Omega$ such that $V:={\rm supp\,}\xi\Subset U$.
For  $\ez\in (0,\,\ez_U)$, let $u^\ez\in C^\fz(U)$ be a solution to \eqref{infty ezf}.
By Lemma \ref{strong converge}, we have
$$ \int_{U}\langle Du , D u  -DP \rangle ^2|Du |^6\xi^2  \,dx
=\lim_{\ez\to0}\int_{U}\langle Du^{\ez}, D u^\ez -DP \rangle ^2|Du^\ez|^6\xi^2  \,dx.$$
By Lemma \ref{DDuez}(i), we have
\begin{align*}\liminf_{\ez\to0}
\int_{U}   |D^2u^\ez Du^\ez|^2   \xi^2 \,dx &\le C\liminf_{\ez\to0}\left[
\int |Du^\ez|^4(|D\xi|^2+|D^2\xi||\xi|)\,dx+C   \left|\int_{U}
f^\ez _iu^\ez _i\xi^2  \,dx\right|\right]\\
&=C\int |Du |^4(|D\xi|^2+|D^2\xi||\xi|)\,dx+C   \left|\int_{U}
f _iu _i\xi^2  \,dx\right| .
 \end{align*}
Using these, Lemma \ref{flat} and Lemma \ref{strong converge}, we obtain
\begin{align*}
 &\int_{U}\langle Du , D u  -DP \rangle ^2|Du |^6\xi^2  \,dx \\
 &\quad=\lim_{\ez\to0}\int_{U}\langle Du^{\ez}, D u^\ez -DP \rangle ^2|Du^\ez|^6\xi^2  \,dx\\
& \quad\le \liminf_{\ez\to0} \left[ \int_{U}   |D^2u^\ez Du^\ez|^2   \xi^2 \,dx \right]^{1/2}
\left[ \int_{U}    |Du^{\ez}|  ^{12}       |Du-DP|^2    (u^\ez-P)^2\xi^2 \,dx\right]^{1/2}\\
&\quad\quad+C  \liminf_{\ez\to0}\int_{U}    [|Du^{\ez}|  ^8(|D\xi|^2 +|D^2\xi||\xi|)+|Du^{\ez}|  ^2 (f^\ez)^2\xi^2](u^\ez-P)^2  \,dx\\
&\quad\quad +C  \liminf_{\ez\to0} \left|\int_{U}
f^\ez_iu^\ez_i|Du^{\ez}|  ^4      (u^\ez-P)^2\xi^2  \,dx\right|\\
& \quad\le  \left[\int |Du |^4(|D\xi|^2+|D^2\xi||\xi|)\,dx+ \left|\int_{U}
f _iu _i\xi^2  \,dx\right|\right]^{1/2}\\
&\quad\quad\quad\times
\left[ \int_{U}    |Du |  ^{12}       |Du-DP|^2    (u -P)^2\xi^2 \,dx\right]^{1/2}\\
&\quad\quad+C  \int_{U}    [|Du |  ^8(|D\xi|^2 +|D^2\xi||\xi|)+|Du |  ^2 (f )^2\xi^2](u -P)^2  \,dx\\
&\quad\quad +C   \left|\int_{U}
f _iu _i|Du |  ^4      (u -P)^2\xi^2  \,dx\right|
 \end{align*}
 as desired.
\end{proof}

\section{Proof of Theorem \ref{mainthm}  when  $ f\in  BV_\loc(\Omega)\cap C^0(\Omega)$}

Suppose $\Omega \Subset\mathbb R^2$,   and $ f\in BV_\loc(\Omega)\cap C^0(\Omega)$ with  $f>0$ in $\Omega$.
Given any domain $U \Subset\Omega$,  let  $ \ez_U $, $ \wz U$ and $ f^\ez$ for $ \ez\in (0,\ez_U]$ be as in Section 3.
 For $\dz\in(0, \ez_U )$, we have
  $$\|  f^\dz\|_{C^0(\overline U)}\le \|f\|_{C^0(\overline{\wz U})}\quad{\rm and }\quad \|  f^\dz\|_{BV(U)}\le  \|f\|_{BV(\wz U)}.$$
 For any ball $B=B(x,R)\subset  U$ with radius $R$, if $\dz\le \min\{R,\ez_U\}$, we have
  $$\|  f^\dz\|_{C^0(\overline B)}\le \|f\|_{C^0(2\overline B)}\quad{\rm and }\quad \|  f^\dz\|_{BV(B)}\le  \|f\|_{BV(2B)}.$$

For each  $  \dz \in(0,\ez_U]$,   let $ \hat u^\dz\in   C(\overline U)$ be a   solution to
\begin{equation}\label{infty dzf}
-\Delta_{\infty} \hat u^\dz =f^\dz \quad \text{in $U$};\quad u^\ez=u   \quad \text{on $\partial  U$}.
 \end{equation}
Since $f^\dz\in C^\fz(U)$ and $f^\dz>0$ in $U$, as proved in Section 4, Theorem \ref{mainthm} and Lemma \ref{flatu} hold  for $\hat u^\dz$ in $U$.
By Lemma \ref{stability}, we know that  $ \hat u^\dz\to u$ in $C^{0}(\overline U)$. Below we obtain a Sobolev convergence, which is crucial to prove Theorem \ref{mainthm} when    $0<f\in  BV_\loc(\Omega)\cap C^0(\Omega)$.

\begin{lem}\label{strong stability}
\begin{enumerate}
\item[(i)]
For $\alpha>3/2$, we have
 $|D  \hat u^\dz|^{\alpha}\to  |D u |^{\alpha}$ in $L^{ p}_\loc(U)$ for all $p\ge1$ and weakly in $W^{1,2}_\loc(U)$.
 Moreover,   $\hat u^\dz\to u$ in $W^{1,p}_\loc(U)$  for all $p\ge1$.

\item[(ii)] For $\alpha\in(0,3/2]$ and $p\in[1, 3/(3-\alpha))$, we have
 $|D  \hat u^\dz|^{\alpha}\to  |D u |^{\alpha}$ in $L^{ t}_\loc(U)$ for all $t\ge1$ and weakly in $W^{1,p}_\loc(U)$.
 \end{enumerate}
 \end{lem}

To prove this lemma, we need the following uniform Sobolev estimates.
 \begin{lem} \label{uniform DDudz}

\begin{enumerate}
\item[(i)]
  For $\alpha>3/2$, we have
    $|D \hat u^\dz|^{\alpha} \in  W^{1,2}_\loc(U)$ uniformly in $\dz\in(0,\ez_U)$.

  \item[(ii)]   For $\alpha\in(0,3/2]$ and $p\in(1,3/(3-\alpha))$, we have $|D \hat u^\dz|^{\alpha} \in  W^{1,p}_\loc(U)$ uniformly in $\dz\in(0,\ez_U)$.
 \end{enumerate}
   \end{lem}

   \begin{proof} Notice that by Remark 1.3 (ii) to $\hat u^\dz$, we know that here (ii) follows from (i).
   Below, we prove (i).
  For $\alpha>3/2$, by Lemma \ref{lipu} and Lemma \ref{max}, we know that
  for each ball $B\subset 2B \Subset U$ with radius $R$ and $\dz<\ez_U$,
\begin{align*}
\|D\hat u^\dz\|_{L^\fz( B)}&\le  \frac CR\|\hat u^\dz\|_{C^0(2\overline B)}+C(R\|f^\dz\|_{C^0(2\overline B)})^{1/3} \le  \frac CR\|  u \|_{C^0(\overline U)}+  C(R\|f \|_{C(\overline {\wz U})} )^{1/3}.
  \end{align*}
Therefore,    for each ball $B\subset 4B \Subset U$ with radius $R$ and $\dz<\ez_U$,  and for any $\xi\in C_c^1(2B) $  with $0\le \xi\le1$, we have
\begin{align}\label{dudf}
\int_U f^\dz_i \hat u^\dz_i|D\hat u^\dz|^{2\alpha-4}\xi\,dx&\le
\|f^\dz\|_{BV(2B)}\|D\hat u^\dz\|^{2\alpha-3}_{L^\fz(2B)}\nonumber\\
&\le C\|  f \|_{BV(\wz U)}\left[\frac1R\| u \|_{C^0(\overline U)}+ (R\|f \|_{C(\overline{\wz U})})^{1/3}\right].
\end{align}
 Applying  to $f^\dz, \hat u^\dz$ with   $\xi\in C^2_c(2 B)$ satisfying $\xi=1$ on $B$, $0\le \xi\le1$ and
$|D\xi| +|D^2\xi|^{1/2}\le\frac CR$,    we obtain
\begin{align*}
  \int_B  |D |D \hat u^\dz|^{\alpha}  |^2   \,dx &\le \int_U  |D |D \hat u^\dz|^ {\alpha}  |^2 \xi^2  \,dx\\
  &\le C ({ \alpha})   \int_{U}|D  \hat u^{\dz}| ^{ 2\alpha } (|D\xi|^2 +|D^2\xi||\xi|)\,dx + C(\alpha)\left| \int_{U}    \hat u^\dz_i    f^\dz_i |D\hat u^\dz|^{2\alpha-4} \xi^2  \, dx \right| \\
  &\le C ({ \alpha}) \|D  \hat u^{\dz}\| ^{ 2\alpha } _{L^\fz(2B)}+C(\alpha)\|f^\dz\|_{BV(2B)}\|D\hat u^\dz\|^{2\alpha-3}_{L^\fz(2B)}
\end{align*}
which is then bounded uniformly in  $\dz\in (0,\ez_U)$, that is,   $|D \hat u^\dz|^{\alpha} \in  W^{1,2}_\loc(U)$ uniformly in $\delta>0$.
\end{proof}

The idea of the proof of Lemma \ref{strong stability} is  similar to that of Lemma \ref{strong converge}.

 \begin{proof}[Proof of Lemma \ref{strong stability}]

  By Lemma \ref{uniform DDudz} (i) for $\alpha=2$ we know that
$D|D\hat u^\dz| ^2\in W^{1,2}_\loc(U) $ uniformly in $\dz$.
From the weak compactness of $W_\loc^{1,2}(U)$, it follows that
$|D\hat u^\dz| ^2 $ converges, up to some subsequence,  to some function $\hat h  $ in $L^t_\loc(U)$ for all $t\ge1$
and weakly  in $W^{1,2}_\loc(U)$ as $\dz\to0$.
By Lemma \ref{uniform DDudz} again and a similar reason as in the proof of Lemma \ref{strong converge}, the proof of Lemma \ref{strong stability} is reduced to proving  $\hat h =|Du|^2 $   almost everywhere. Here we omit the details.


Below we  prove  $\hat h(\bar x) =|Du(\bar x)|^2 $ for all $\bar x\in U$ satisfying that
 $u$ is differentiable at $\bar x$, and  $\bar x$ is Lebesgue point of $[h]^{14}$  and $ Du $.
Note that the set of such $\bar x$ has full measure in $U$.

If $h(\bar x)=0$, similarly to the proof of Lemma \ref{strong converge}, we have $|Du(\bar x)|^2=0$.

Assume that $h(\bar x) >0$ below.
For any $\lambda\in(0,1)$, thanks to the differentiability at $\bar x$ of $u$,
there exists  $r_{\lambda,\bar x}\in (0, \dist(\bar x, \partial U)/8)$ such that for any $ r\in(0,r_{\lambda,\bar x})$, we have
$$ \sup_{B(\bar x, 2r)}\frac{|u(x) - u(\bar x)-\langle Du(\bar x),(x-\bar x)\rangle | }{r }\le \lambda.$$
By Lemma \ref{stability}, for arbitrary     $ r\in(0,r_{\lambda,\bar x})$, there exists $\dz_{\lambda,\bar x, r}\in(0, r]$ such that for all $\dz\in(0,\dz_{\lambda,\bar x,r })$, we have
$$ \sup_{B(\bar x, 2r)}\frac{|\hat u^\dz(x) -  \hat u^\dz(\bar x)-\langle Du(\bar x),(x-\bar x)\rangle| }{r }\le 2\lambda.$$
By the same argument as in the proof of Lemma \ref{strong converge}, to obtain
$h (\bar x)=|Du(\bar x)|^2 $,
it suffices to prove that
for all $r\in(0,r_{\lambda,\bar x})$ and $ \dz\in(0,\dz_{\lambda,\bar x,r})$,
\begin{align}\label{flatudz}
 \bint_{B(\bar x,r)} (|Du ^{\dz}|^2- \langle Du(\bar x),D\hat u^\dz\rangle  )^2|Du ^{\dz}|^6 \,dx\le
 C (\bar x, \|u\|_{C^0(\overline U)},\|f\|_{C^0(\overline U)},\|Df\|_{BV(U)}) \lambda .
\end{align}
We omit the details here.

To prove \eqref{flatudz},    applying  Lemma \ref{flatu} to $\hat u^\dz$ and
 $P(x)= \hat u^\dz (\bar x)-\langle Du(\bar x),(x-\bar x)\rangle$,
    we get
 \begin{align*}
& \int_{U}\langle D\hat u^\dz  , D u ^\dz -DP \rangle ^2|D\hat u^\dz |^6  \xi^2\,dx\\
& \le C\left  [  \int_{U}|Du ^\dz| ^4 |D\xi|^2\,dx +  \left|\int_{U}   \hat u^\dz_if^\dz_i \xi^2  \, dx \right|  \right]^{1/2}
\left [ \int_{U}    |D\hat u^\dz |  ^{12}       |D\hat u^\dz-DP|^2    (\hat u^\dz -P)^2\xi^2 \,dx\right ]^{1/2}\\
&\quad+C  \int_{U}    [|D\hat u^\dz |  ^8(|D\xi|^2 +|D^2\xi||\xi|) +|D\hat u^\dz |  ^2 (f^\dz )^2\xi^2] (\hat u^\dz-P)^2  \,dx\\
&\quad+C\left|\int_U f^\dz_i \hat u^\dz_i|D\hat u^\dz|  ^4 \xi^2   (\hat u^\dz-P)^2  \,dx\right|.
 \end{align*}

For any $B=B(\bar x, r)$ with $r\in(0,r_{\lambda,\bar x})$ and $ \dz\in(0,\epsilon_{\lambda,\bar x,r})$
 taking suitable cut-off function   $\xi\in C^2_c(2  B)$ satisfying $\xi=1$ on $B$, $0\le \xi\le1$ and
$|D\xi| +|D^2\xi|^{1/2}\le\frac Cr$.
Then
$$\int_{B}\langle D\hat u^\dz  , D u ^\dz -DP \rangle ^2|D\hat u^\dz |^6  \,dx\le
\int_{U}\langle D\hat u^\dz  , D u ^\dz -DP \rangle ^2|D\hat u^\dz |^6  \xi^2\,dx.$$

Moreover, for the first term in the right hand side,
 by Lemma \ref{lipu}  we obtain
  \begin{align*}
 & \int_{U}|D\hat u^\dz| ^4 |D\xi|^2\,dx +  \int_{U}   \hat u^\dz_if^\dz_i \xi^2  \, dx\\
  &\quad\le C\bint_{2B}|D\hat u^\dz| ^4\,dx+ Cr^2\|f^\dz\|_{BV(2B)}\|D\hat u^\dz\|_{L^\fz(2B)} \\
  &\quad\le C \|D\hat u^\dz\| _{L^\fz(B(\bar x, \frac14\dist(\bar x,\partial U)))}[1+ Cr^2\|\hat f \|_{BV(U)}]\\
  &\quad\le C   \left[ \frac1{\dist(\bar x,\partial U)}\|u\|_{C^0(\overline U )}+ (\dist(\bar x,\partial U)\|f\|_{C^0(\overline U)})^{1/3}\right]^4\\
  &\quad\quad +
   C \dist(\bar x,\partial U)\|\hat f \|_{BV(U)}\left[ \frac1{\dist(\bar x,\partial U)}\|u\|_{C^0(\overline U )}+ (\dist(\bar x,\partial U)\|f\|_{C^0(\overline U)})^{1/3}\right],
       \end{align*}
and
    \begin{align*}
&  \int_{U}    |D\hat u^\dz |  ^{12}       |D\hat u^\dz-DP|^2    (\hat u^\dz -P)^2\xi^2 \,dx\\&\quad\le
   r^4\lz^2\|D\hat u^\dz\|^{12}_{L^\fz(2B)} [\|D\hat u^\dz\|^2_{L^\fz(2B)} +|Du(\bar x)|^2]\\
    &\quad\le C r^4\lz^2   \left  [ \frac1{\dist(\bar x,\partial U)}\|u\|_{C^0(\overline U )}+ (\dist(\bar x,\partial U)\|f\|_{C^0(\overline U)})^{1/3}\right] ^{14}.
      \end{align*}
For the second and third terms in the right hide side, similarly we have
\begin{align*}
& \int_{U}    [|D\hat u^\dz |  ^8(|D\xi|^2 +|D^2\xi||\xi|) +|D\hat u^\dz |  ^2 (f )^2\xi^2 ]     (\hat u^\dz-P)^2  \,dx \\
&\quad\le Cr^2\lz^2 [\|D\hat u^\dz\|^8_{L^\fz(2B)}+ r^2\|f^\dz\|^2_{C^0(2\overline B)}\|D\hat u^\dz\|^2_{L^\fz(2B)} ]\\
&\quad\le C r^2\lz^2\left\{\left[ \frac1{\dist(\bar x,\partial U)}\|u\|_{C^0(\overline U )}+ (\dist(\bar x,\partial U)\|f\|_{C^0(\overline U)})^{1/3}\right] ^8\right.\\
&\quad\quad+\left.
 \dist(\bar x,\partial U)^2\|\hat f \|^2_{C^0(\overline U)}
\left[ \frac1{\dist(\bar x,\partial U)}\|u\|_{C^0(\overline U )}+ (\dist(\bar x,\partial U)\|f\|_{C^0(\overline U)})^{1/3}\right] ^2
\right\} ,
 \end{align*}
 and
  \begin{align*}
 &\left|\int_U f^\dz_i \hat u^\dz_i|D\hat u^\dz|  ^4 \xi^2   (\hat u^\dz-P)^2  \,dx\right|\\
 &\quad\le C r^2\lz^2\|Df^\dz\|_{BV(2B)} \|D\hat u^\dz\|^5_{L^\fz(2B)}\\
 &\quad\le C r^2\lz^2 \|Df^\dz\|_{BV(U)}\left[ \frac1{\dist(\bar x,\partial U)}\|u\|_{C^0(\overline U )}+ (\dist(\bar x,\partial U)\|f\|_{C^0(\overline U)})^{1/3}\right] ^5.
  \end{align*}
Combining all estimates together we have \eqref{flatudz} as desired.
 This completes the proof of Lemma  \ref{strong stability}.
 \end{proof}

Finally, we prove Theorem \ref{mainthm} when $0<f\in BV_\loc(\Omega)\cap C^0(\Omega)$ as below.

\begin{proof}[Proof of Theorem \ref{mainthm} when $0<f\in BV_\loc(\Omega)\cap C^0(\Omega)$.]

By Lemma \ref{strong stability}, we have
$|Du|^{\alpha}\in W^{1,2}_\loc(U)$ when $\alpha>3/2$, and $|Du|^{\alpha}\in W^{1,p}_\loc(U)$ when $\alpha\in(0,3/2]$ and $p\in[1,3/(3-\alpha))$.  This gives (ii), and also reduces
(i) to verifying   \eqref{ex01}.
Moreover, by Remark \ref{relthm} (i), we know that (iii) follows from (i) and (iv).
 So below we only need to prove   \eqref{ex01}, \eqref{strongest} when $f\in W^{1,1}_\loc(\Omega)$ additionally,  and (iv), that is, \eqref{pw} and \eqref{chain}.
Recall that, as proved in Section 4,  Theorem \ref{mainthm} holds for $\hat u^\dz$ for any $\dz\in(0,\ez_U)$.

{\it Proof of \eqref{strongest} when $f\in W^{1,1}_\loc(\Omega)$ additionally.}
By $\hat u^\dz\to u$ in $W^{1,p}_\loc(U)$ and $|D\hat u^\dz|^{2\alpha-4} \to |D u |^{2\alpha-4} $ in $L^p_\loc(U)$ for any $p\ge1$ as $\dz\to0$
 implies that $|D\hat u^\dz|^{2\alpha-3}D\hat u^\dz\to |D u |^{2\alpha-3}Du$ in weak-$\ast$ topology of $L^1_\loc(U)$.
Observing $f^\dz\to f$ in $W^{1,1}_\loc(U)$, we have
$$\int_\Omega f^\dz_i u^\dz_i|D\hat u^\dz|^{2\alpha-4}\xi^2\,dx\to \int_\Omega f_i u_i|Du|^{2\alpha-4}\xi^2\,dx$$
as $\ez\to0$.
Since \eqref{strongest} holds for $\hat u^\dz$,  by Lemma 5.1 and 5.2 we futher obtain \eqref{strongest} for $u$ as desired.

{\it Proof of \eqref{ex01}.} By Lemma \ref{strong stability},
for $\alpha>3/2$ and for all $B\subset 8B\Subset U$   we have
\begin{align*}
 &
 \int_B  |D |D  u |^{\alpha}  |^2   \,dx
 \le \liminf_{\dz\to0}  \int_B  |D |D \hat u^\dz|^{\alpha}  |^2   \,dx\le  \lim_{\dz\to0} \int_U  |D |D \hat u^\dz|^ {\alpha}  |^2 \xi^2  \,dx,\end{align*}
where  $\xi\in C^2_c(2  B)$ satisfying $\xi=1$ on $B$, $0\le \xi\le1$ and
$|D\xi| +|D^2\xi|^{1/2}\le\frac CR$.
Since \eqref{strongest} holds for $\hat u^\dz$,   we have
 \begin{align*}
  &\int_U  |D |D \hat u^\dz|^{\alpha}  |^2\xi^2   \,dx\\
   &\quad\le C ({ \alpha})  \liminf_{\dz\to0} \int_{U}|D  \hat u^{\dz}| ^{ 2\alpha } (|D\xi|^2 +|D^2\xi||\xi|)\,dx + C(\alpha) \liminf_{\dz\to0}\int_{U}    \hat u^\dz_i    f^\dz_i |D\hat u^\dz|^{2\alpha-4} \xi^2  \, dx  \\
  &\quad\le C ({ \alpha}) \liminf_{\dz\to0}\bint_{2B}|D  \hat u^{\dz} | ^{ 2\alpha }\,dx  +C(\alpha)\liminf_{\dz\to0}\|f^\dz\|_{BV(2B)}\|D\hat u^\dz\|^{2\alpha-3}_{L^\fz(2B)}
\end{align*}
Notice that when $\delta>0$ is sufficient small,
by Lemma \ref{lipu} we have
$$\|D\hat u^\dz\| _{L^\fz(2B)}\le  C\frac1R\| u^\dz\|_{C^0(2\overline B)}+(R\|f^\dz\|_{C^0( 2\overline B)})^{1/3}
\le C\frac1R\| u\|_{C^0(2\overline B)}+C(R\|f\|_{C^0( 4\overline B)})^{1/3}.$$
By  Lemma \ref{strong stability} again, we   conclue
 \begin{align*}
   \int_B  |D |D  u|^{\alpha}  |^2   \,dx &\le C ({ \alpha})\bint_{2B}|D  u | ^{ 2\alpha }\,dx\\
   &\quad +
     C(\alpha)\|f \|_{BV(4B)}[ \frac1R\| u\|_{C^0(4\overline B)}+(R\|f\|_{C^0( 4\overline B)})^{1/3}]^{2\alpha-3} ,
\end{align*}
as desired.

{\it Proof of \eqref{chain}.}
By Lemma \ref{strong stability} and applying \eqref{chain} to $\hat u^\dz$ we have, for any $\Phi\in C^\fz_c(U,\mathbb R^2)$, \begin{align*}
\int_U |Du|^\tau \langle D|Du|^\alpha, \Phi\rangle \,dx&= \lim_{\dz\to0}
\int_U  |D\hat u^\dz|^ \tau  \langle D|D\hat u^\dz|^\alpha,\Phi\rangle \,dx\\
&=\frac{\alpha}{\alpha+\tau} \lim_{\dz\to0}
\int_U  \langle D|D\hat u^\dz|^{\alpha+\tau},  \Phi\rangle \,dx\\
&=\frac{\alpha}{\alpha+\tau}
\int_U  \langle D |Du | ^{ \alpha+\tau },  \Phi\rangle \,dx,
\end{align*}
which gives \eqref{chain}.

{\it Proof of   \eqref{pw}.} Given any $U\Subset\Omega$, by  Lemma \ref{strong stability}  and applying  (iv) to $\hat u^\dz$, we have
\begin{align*}\int_U \langle D|Du|^2,Du\rangle \phi\,dx&=
\lim_{\dz\to0} \int_U \langle D|D\hat u^\dz|^2,D\hat u^\dz\rangle \phi\,dx\\
& = \lim_{\dz\to0} -  2\int_U      f^\dz \phi\,dx =
  -  2\int_U        f  \phi\,dx \quad \forall \phi\in C_c^\fz(U),\end{align*}
  which implies that  $(|Du|^2)_iu_i=-2f$ almost everywhere in $U$ and hence in $\Omega$.
  Note that $|Du| $ vanishes only in a set with measure $0.$
  By this and \eqref{chain}, for $\alpha>2$ or $\alpha\in(0,2)$ we have  $(|Du|^\alpha)_i=\frac\alpha2 |Du|^{ \alpha-2 } (|Du|^2)_i$ almost everywhere, and hence $(|Du|^\alpha)_iu_i= -\alpha |Du|^{ \alpha-2 }f$ almost everywhere in $\Omega.$
  This completes the proof of Theorem \ref{mainthm}.
\end{proof}

\section{Proofs of Lemma \ref{DDuez}, Lemma \ref{Duez} and Lemma \ref{flat}}

We first derive the following identity by taking    $\phi=|Du^\ez|^2(|Du^\ez|^2+\kappa)^{   { \alpha-2}} \xi^2$ in \eqref{functional}
and  applying Lemma \ref{identity1}.

\begin{lem} \label{identity2} Let $\xi\in C^2_c(U)$.
If $ \alpha \ge 2$ and $\kappa\ge 0$ or $ \alpha >0$ and $\kappa>0$, we have
\begin{align*} &\int_{U}  |D^2u^\ez Du^\epsilon |^2  (|Du^\ez|^2+\kappa)^{    \alpha-3} [2(|Du^\ez|^2+\kappa)+({ \alpha-2}) |Du^\ez|^2]\xi^2 \, dx \\
&\quad\quad
+ \ez\int_{U}  | \Delta u^\epsilon |^2  (|Du^\ez|^2+\kappa)^{ \alpha-3} [2(|Du^\ez|^2+\kappa)+({ \alpha-2}) |Du^\ez|^2]\xi^2 \, dx\\
&\quad\quad+\int_{U}\langle Du^{\ez}, D\xi\rangle ^2|Du^\ez|^2 (|Du^\ez|^2+\kappa)^{   { \alpha-2}} \,dx\\
&\quad= -2\int_{U}  (|Du^\ez|^2+\kappa)^{    \alpha-3}[(|Du^\ez|^2+\kappa)+({ \alpha-2})|Du^{\ez}|  ^2]     u^{\ez}_i \xi_i  \Delta_\fz u^{\ez}  \xi \,dx\\
&\quad\quad   -  \int_{U}\xi_{ik} u^{\ez}_k u^{\ez}_i |Du^\ez|^2 (|Du^\ez|^2+\kappa)^{   { \alpha-2}}\xi  \,dx\nonumber\\
&\quad\quad-2\int_{U} |Du^{\ez}| ^2 (|Du^\ez|^2+\kappa)^{   { \alpha-2}}  u^{\ez}_{ik}u^{\ez}_k\xi_i \xi \,dx\nonumber \\
&\quad\quad -\int_{U}
 f^\ez \Delta u^\ez  (|Du^\ez|^2+\kappa)^{    \alpha-3}[2(|Du^\ez|^2+\kappa)+({ \alpha-2}) |Du^\ez|^2] \xi^2 \, dx
\end{align*}
\end{lem}

\begin{proof}
  Let $\psi=   (|Du^\ez|^2+\kappa)^{   { \alpha-2}} \xi^2$ for
 $ \xi\in C_c^\infty(U)$. Then $\phi=\psi|Du^\ez|^2,\psi\in W^{1,\,2}_c(U)$.
By \eqref{identityIII}, we write
\begin{align*}\mathbb I_\epsilon(\phi)&= \int_{U}  |D^2u^\ez Du^\epsilon |^2  (|Du^\ez|^2+\kappa)^{   { \alpha-2}} \xi^2\, dx +\ez\int_{U}   (\Delta u^\ez)^2 (|Du^\ez|^2+\kappa)^{   { \alpha-2}} \xi^2 \, dx\\
&\quad+  \int_{U}    f^\ez \Delta u^\ez  (|Du^\ez|^2+\kappa)^{   { \alpha-2}} \xi^2 \, dx.
\end{align*}

On the other hand, note that $$  \phi_i=2  u^{\ez}_{ik}u^{\ez}_k   (|Du^\ez|^2+\kappa)^{    \alpha-3}  [(|Du^\ez|^2+\kappa)+({ \alpha-2})|Du^{\ez}|  ^2] \xi^2  +  2\xi  \xi_i  |Du^{\ez}| ^{2 }(|Du^\ez|^2+\kappa)^{   { \alpha-2}} .$$
Pluging in  $ \phi _i$ in \eqref{functional}, we have
\begin{align*}
 \mathbb I_\epsilon(\phi)&= \frac12\int_{U} [\Delta u^\ez u^\ez_i\phi_i  - u^\ez_{ij} u^\ez_j \phi_i]\,dx \\
& = -   \int_{U}    |D^2 u^{\ez} Du^{\ez}|^2  (|Du^\ez|^2+\kappa)^{    \alpha-3} [(|Du^\ez|^2+\kappa)+({ \alpha-2})|Du^{\ez}|  ^2] \xi^2\,dx\\
 &\quad- \int_{U} |Du^{\ez}|  ^2(|Du^\ez|^2+\kappa)^{   { \alpha-2}} u^{\ez}_{ij}u^{\ez}_j\xi_i \xi \,dx\\
&\quad +  \int_{U} (|Du^\ez|^2+\kappa)^{    \alpha-3} [(|Du^\ez|^2+\kappa)+({ \alpha-2})|Du^{\ez}|  ^2]
\Delta u^{\ez} \Delta_\fz u^{\ez}  \xi^2 \,dx\\
&\quad+  \int_{U}|Du^{\ez}|  ^2(|Du^\ez|^2+\kappa)^{   { \alpha-2}}\Delta u^{\ez} u^{\ez}_i\xi_i   \xi  \, dx.
\end{align*}
Replacing $\Delta_\fz u^\ez $ by $-\epsilon\Delta u^\ez-f^\ez$ in  third term, we further have
\begin{align*}
&\int_{U} (|Du^\ez|^2+\kappa)^{    \alpha-3} [(|Du^\ez|^2+\kappa)+({ \alpha-2})|Du^{\ez}|  ^2] \Delta u^{\ez}  \Delta_\fz u^{\ez}  \xi^2 \\
&\quad=- \ez\int_{U} (|Du^\ez|^2+\kappa)^{    \alpha-3}  [(|Du^\ez|^2+\kappa)+({ \alpha-2})|Du^{\ez}|  ^2] (\Delta u^{\ez})^2  \xi^2 \,dx\\
&\quad\quad- \int_{U} (|Du^\ez|^2+\kappa)^{    \alpha-3} [(|Du^\ez|^2+\kappa)+({ \alpha-2})|Du^{\ez}|  ^2]   f^\ez\Delta u^{\ez}   \xi^2 \,dx.
\end{align*}
 Therefore

\begin{align*}
&    \int_{U}    |D^2 u^{\ez} Du^{\ez}|^2  (|Du^\ez|^2+\kappa)^{    \alpha-3} [2(|Du^\ez|^2+\kappa)+({ \alpha-2})|Du^{\ez}|  ^2] \xi^2\,dx\\
&\quad\quad+\ez\int_{U} (|Du^\ez|^2+\kappa)^{    \alpha-3}  [2(|Du^\ez|^2+\kappa)+({ \alpha-2})|Du^{\ez}|  ^2] (\Delta u^{\ez})^2  \xi^2 \,dx\\
 &\quad=- \int_{U} |Du^{\ez}|  ^2(|Du^\ez|^2+\kappa)^{   { \alpha-2}} u^{\ez}_{ij}u^{\ez}_j\xi_i \xi \,dx+  \int_{U}|Du^{\ez}|  ^2(|Du^\ez|^2+\kappa)^{   { \alpha-2}}\Delta u^{\ez} u^{\ez}_i\xi_i   \xi  \, dx\\
&\quad\quad- \int_{U} (|Du^\ez|^2+\kappa)^{    \alpha-3} [2(|Du^\ez|^2+\kappa)+({ \alpha-2})|Du^{\ez}|  ^2]   f^\ez\Delta u^{\ez}   \xi^2 \,dx.
\end{align*}

Via integration by parts   we have
\begin{align}
 &- \int_{U} |Du^{\ez}|  ^2(|Du^\ez|^2+\kappa)^{   { \alpha-2}} u^{\ez}_{ij}u^{\ez}_j\xi_i \xi \,dx+  \int_{U}|Du^{\ez}|  ^2(|Du^\ez|^2+\kappa)^{   { \alpha-2}}\Delta u^{\ez} u^{\ez}_i\xi_i   \xi  \, dx \nonumber  \\
 &
 \quad=-  \int_{U}  u^{\ez}_k (|Du^{\ez}|  ^2(|Du^\ez|^2+\kappa)^{   { \alpha-2}}  u^{\ez}_i\xi_i   \xi )_k \, dx- \int_{U} |Du^{\ez}|  ^2(|Du^\ez|^2+\kappa)^{   { \alpha-2}} u^{\ez}_{ij}u^{\ez}_j\xi_i \xi^3\,dx\nonumber\\
& \quad=  -2\int_{U}  (|Du^\ez|^2+\kappa)^{    \alpha-3}[(|Du^\ez|^2+\kappa)+({ \alpha-2})|Du^{\ez}|  ^2]     u^{\ez}_i \xi_i \Delta_\fz u^{\ez}  \xi \,dx\nonumber\\
 & \quad\quad-\int_{U}\langle Du^{\ez}, D\xi\rangle ^2|Du^\ez|^2 (|Du^\ez|^2+\kappa)^{   { \alpha-2}} \,dx\nonumber\\
&\quad\quad   -  \int_{U}\xi_{ik} u^{\ez}_k u^{\ez}_i |Du^\ez|^2 (|Du^\ez|^2+\kappa)^{   { \alpha-2}}  \,dx\nonumber\\
&\quad\quad-2\int_{U} |Du^{\ez}| ^2 (|Du^\ez|^2+\kappa)^{   { \alpha-2}}  u^{\ez}_{ik}u^{\ez}_k\xi_i \xi \,dx\nonumber
\end{align}
 as desired.
\end{proof}

As a consequence of Lemma \ref{identity2}, we have
\begin{cor} \label{identity3}
  Given any $\alpha>0$ and $0<\eta<<1$, for any $\xi\in C_c(U)$ and  $\kappa\ge 0$, we have
\begin{align*} & (1-\eta)\int_{U}  |D^2u^\ez Du^\epsilon |^2  (|Du^\ez|^2+\kappa)^{    \alpha-3} [ (|Du^\ez|^2+\kappa)+({ \alpha-2}) |Du^\ez|^2]\xi^2 \, dx \\
&\quad\quad
+ \ez\int_{U}  | \Delta u^\epsilon |^2  (|Du^\ez|^2+\kappa)^{ \alpha-3} [   (|Du^\ez|^2+\kappa)+ ({ \alpha-2})  |Du^\ez|^2]\xi^2 \, dx\\
&\quad\le C(\alpha, \eta )  \int_{U}  (|Du^\ez|^2+\kappa)^{    \alpha  } (|D\xi|^2 +|D^2\xi||\xi|)\,dx\nonumber\\
&\quad\quad-\int_{U}
 f^\ez \Delta u^\ez  (|Du^\ez|^2+\kappa)^{    \alpha-3}[2(|Du^\ez|^2+\kappa)+({ \alpha-2}) |Du^\ez|^2] \xi^2 \, dx.
\end{align*}
\end{cor}

\begin{proof}
It suffices to estimate the first three terms in the right hand side of the identity given in Lemma \ref{identity2}.
Obviously,
$$-  \int_{U}\xi_{ik} u^{\ez}_k u^{\ez}_i |Du^\ez|^2 (|Du^\ez|^2+\kappa)^{   { \alpha-2}}\xi  \,dx\le
\int_U(|Du^\ez|^2+\kappa)^{  \alpha +1}|D^2\xi||\xi|\,dx.$$
By Young's inequality, we have
\begin{align*}
 &-2\int_{U} |Du^{\ez}| ^2 (|Du^\ez|^2+\kappa)^{   { \alpha-2}}  u^{\ez}_{ik}u^{\ez}_k\xi_i \xi \,dx\\
 &\quad\le
 \frac\eta2  \int_{U} |D^2u^\ez Du^\epsilon |^2   (|Du^\ez|^2+\kappa)^{ { \alpha-2}}    \xi^2 \,dx+C(\eta)
 \int_{U}   (|Du^\ez|^2+\kappa)^{ \alpha}    |D\xi|^2 \,dx.
 \end{align*}
 Write
 $$K=   -2\int_{U}  (|Du^\ez|^2+\kappa)^{    \alpha-3}[(|Du^\ez|^2+\kappa)+({ \alpha-2})|Du^{\ez}|  ^2]     u^{\ez}_i \xi_i  \Delta_\fz u^{\ez}  \xi \,dx.$$
If ${ \alpha }\ge2$, by Young's inequality, we have
\begin{align*}
 K
&\le 2\int_{U} |D^2u^\ez Du^\epsilon |    (|Du^\ez|^2+\kappa)^{   { \alpha-2}}[(|Du^\ez|^2+\kappa)+({ \alpha-2})|Du^{\ez}|  ^2]    | D \xi|  |\xi | \,dx\\
&\le \frac \eta2\int_{U}  |D^2u^\ez Du^\epsilon |^2  (|Du^\ez|^2+\kappa)^{    \alpha-3} [ (|Du^\ez|^2+\kappa)+({ \alpha-2}) |Du^\ez|^2]\xi^2 \, dx\\
&\quad +
C( \eta)({ \alpha-1})
 \int_{U}   (|Du^\ez|^2+\kappa)^{ \alpha}    |D\xi|^2 \,dx
\end{align*}
If $  \alpha <2$, by Young's inequality, we have
\begin{align*}
K
&\le 4\int_{U} |D^2u^\ez Du^\epsilon |    (|Du^\ez|^2+\kappa)^{   { \alpha-1}}    | D \xi|  |\xi| \,dx\\
&\le \frac \eta2\int_{U}  |D^2u^\ez Du^\epsilon |^2  (|Du^\ez|^2+\kappa)^{   { \alpha-2}}  \xi^2 \, dx +C( \eta)  \int_{U}   (|Du^\ez|^2+\kappa)^{ \alpha}    |D\xi|^2 \,dx
\end{align*}
as desired. \end{proof}

Now we prove Lemma \ref{DDuez} via Corollary \ref{identity3}.

\begin{proof}[Proof of Lemma \ref{DDuez}.]
We consider the four cases separately.

\medskip
{\it Proof of (i): Case $\alpha=2$.}
By Lemma \ref{identity3} with $\alpha=2$ and $ \kappa=0$, have
\begin{align*} &(2-\eta)\int_{U}  |D^2u^\ez Du^\epsilon |^2 \xi^2\, dx
+ 2\ez\int_{U}  |\Delta u^\epsilon |^2   \xi^4 \, dx
 \le C\int_{U}      |Du^\ez|^4 |D\xi|^2\,dx -2\int_{U}
 f^\ez \Delta u^\ez   \xi^2 \, dx.
\end{align*}
By integration by parts we have
\begin{align*}
 -2\int_{U}
 f^\ez \Delta u^\ez   \xi^2 \, dx
&=  \int_{U}      u^\ez_k [f^\ez   \xi^2]_k \, dx  =  2 \int_{U}  u^\ez_k f^\ez _k \xi^2  \, dx
+4  \int_{U}      u^\ez_k \xi_ k f^\ez \xi    \, dx.
\end{align*}
Observe that $\Delta_\fz u^\ez+\ez\Delta u^\ez=-f^\ez$ implies that
$$
 f^\ez  |Du^\ez| \le  |D^2u^\ez Du^\ez| |Du^\ez|^2+ 2\ez  |\Delta u^\ez||Du^\ez| . $$
We have
\begin{align*}
4\int_{U}      u^\ez_k \xi_ k f^\ez \xi    \, dx
&\le \int_U |D^2u^\ez Du^\ez| |Du^\ez|^2|D\xi|\xi \,dx + 2\ez \int_U |\Delta u^\ez||Du^\ez| |D\xi|\xi \,dx\\
&\le \eta \int_U |D^2u^\ez Du^\ez| ^2\xi^2 \,dx + C(\eta)\int_U  |Du^\ez|^4|D\xi|^2 \,dx \\
&\quad+ \eta\ez\int_U |\Delta u^\ez| ^2\xi^2 \,dx + C(\eta)\ez^2 \int_U    |D\xi|^2 \,dx.
\end{align*}
Therefore, we have
\begin{align*} & \int_{U}  |D^2u^\ez Du^\epsilon |^2 \xi^2\, dx
+  \ez\int_{U}  |\Delta u^\epsilon |^2   \xi^4 \, dx\\
 &\quad\le C\int_{U}      |Du^\ez|^4 |D\xi|^2\,dx +C\int_{U}
 f^\ez_i u^\ez_i  \xi^2 \, dx+C \ez^2 \int_U    |D\xi|^2 \,dx
\end{align*}
as desired.
\medskip

 {\it Proof of (ii): Case ${ \alpha}\ge3$. }
By Lemma \ref{identity3} with $  \alpha \ge3$ and $ \kappa=0$,
 we have
\begin{align*} &\alpha(1-\eta)\int_{U}  |D^2u^\ez Du^\epsilon |^2 |Du^\epsilon |^{2\alpha-4}\xi^4 \, dx
+ \alpha\ez\int_{U}  |D^2u^\ez Du^\epsilon |^2  |Du^\epsilon |^{2\alpha-4}  \xi^2 \, dx\\
&\quad\le  C(\eta)(\alpha-1) \int_{U}      |Du^\ez|^{ 2\alpha }|D\xi|^2 \, dx   -2\alpha\int_{U}
 f^\ez \Delta u^\ez   |Du^\epsilon |^{2\alpha-4} \xi^2 \, dx.
\end{align*}
By integration by parts we have
\begin{align*}
  &-2\alpha\int_{U}
 f^\ez \Delta u^\ez    |Du^\epsilon |^{2\alpha-4}  \xi^2 \, dx\\
&\quad=  2\alpha\int_{U}      u^\ez_k [f^\ez    |Du^\epsilon |^{2\alpha-4}  \xi^2]_k \, dx \\
&\quad =  2\alpha \int_{U}  u^\ez_k f^\ez _k   |Du^\epsilon |^{2\alpha-4} \xi^2  \, dx
+4 \alpha  \int_{U}      u^\ez_k \xi_ k f^\ez   |Du^\epsilon |^{2\alpha-4} \xi  \, dx\\
&\quad\quad+4\alpha(\alpha-2) \int_{U}    f^\ez \Delta_\fz u^\epsilon  |Du^\epsilon |^{2 \alpha-6}   \xi^2  \, dx.
\end{align*}
By Young's inequality, we have $$
4 \alpha  \int_{U}      u^\ez_k \xi_ k f^\ez   |Du^\epsilon |^{2\alpha-4} \xi  \, dx\le C(\eta) \alpha
\int_{U}      |Du^\epsilon |^{ 2\alpha } |D\xi|^2  \, dx+\eta \alpha \int_{U}      (f^\ez)^2|Du^\epsilon |^{2 \alpha-6} \xi^2  \, dx.$$
Applying $\Delta_\fz u^\epsilon=-\Delta  u^\epsilon-f^\ez$ we obtain
\begin{align*}
& 4\alpha(\alpha-2) \int_{U}    f^\ez \Delta_\fz u^\epsilon  |Du^\epsilon |^{2 \alpha-6}   \xi^2  \, dx\\
&\quad=-4\alpha(\alpha-2) \ez \int_{U}    f^\ez \Delta  u^\epsilon  |Du^\epsilon |^{2 \alpha-6}   \xi^2  \, dx- 4\alpha(\alpha-2)  \int_{U}    (f^\ez)^2    |Du^\epsilon |^{2 \alpha-6}   \xi^2  \, dx\\
&\quad\le C(\alpha)\ez \left [\int_{U}     (\Delta  u^\epsilon)^2       \xi^2  \, dx\right]^{1/2}  \left[\int_{U}     (f^\ez )^2 |Du^\epsilon |^{4 \alpha-12}      \xi^2  \, dx\right]^{1/2}\\
&\quad\quad\quad - 4\alpha(\alpha-2)  \int_{U}    (f^\ez)^2  |Du^\epsilon |^{2 \alpha-6}   \xi^2  \, dx.
\end{align*}
Combining all estimates together we arrive at the desired result.

\medskip

{\it Proof of (iii): Case $2< \alpha <3$.}
 By Lemma \ref{identity3} with $\kappa>0$ and $ \alpha >2$,   it suffices to estimate
 $$K=-\int_{U}
  f^\ez \Delta u^\ez  (|Du^\ez|^2+\kappa)^{    \alpha-3}[2(|Du^\ez|^2+\kappa)+({ \alpha-2}) |Du^\ez|^2] \xi^2 \, dx.$$
Via integration by parts, we write
\begin{align*}
K&=-\alpha\int_{U}
 f^\ez \Delta u^\ez  (|Du^\ez|^2+\kappa)^{   { \alpha-2}}  \xi^2 \, dx
 +({ \alpha-2})\kappa \int_{U}
 f^\ez \Delta u^\ez  (|Du^\ez|^2+\kappa)^{    \alpha-3}  \xi^2 \, dx\\
 &= \alpha \int_{U}
   u^\ez_i[ f^\ez  (|Du^\ez|^2+\kappa)^{   { \alpha-2}}  \xi^2]_i \, dx
 -\kappa ({ \alpha-2}) \int_{U}u^\ez_i[ f^\ez  (|Du^\ez|^2+\kappa)^{    \alpha-3}  \xi^2]_i \, dx\\
  &= \alpha \int_{U}
   u^\ez_i  f^\ez_i (|Du^\ez|^2+\kappa)^{   { \alpha-2}}  \xi^2  \, dx
 -\kappa ({ \alpha-2}) \int_{U}u^\ez_i  f^\ez _i (|Du^\ez|^2+\kappa)^{    \alpha-3}  \xi^2  \, dx\\
  &\quad+2\alpha \int_{U}
   u^\ez_i\xi_i  f^\ez  (|Du^\ez|^2+\kappa)^{   { \alpha-2}}  \xi  \, dx
 -2\kappa ({ \alpha-2}) \int_{U}u^\ez_i\xi_i f^\ez  (|Du^\ez|^2+\kappa)^{    \alpha-3}  \xi  \, dx\\
 &\quad+  2 \alpha(\alpha-2) \int_{U}
   \Delta_\fz u^\ez  f^\ez  (|Du^\ez|^2+\kappa)^{    \alpha-3}  \xi^2  \, dx\\
&\quad
 -2\kappa ({ \alpha-2}) ( \alpha-3)\int_{U}\Delta_\fz u^\ez   f^\ez  (|Du^\ez|^2+\kappa)^{    \alpha-4 }  \xi^2  \, dx\\
 &=K_1+\cdots+K_6.
\end{align*}

Notice that
$$ K_2\le
 C  \kappa^{{ \alpha-3/2}} \int_{U}
   |D  f^\ez|    \xi^2  \, dx$$
   and $$ K_4\le C \kappa^{{ \alpha-3/2}} \int_{U}
   |   f^\ez|  |D\xi|  |\xi |  \, dx.$$
By Young's inequality, we have
\begin{align*}
 K_3&\le C \int_{U}
   |D \xi|  f^\ez  (|Du^\ez|^2+\kappa)^{   { \alpha-3/2}}  |\xi | \, dx\\
 &\le C \int_{U}
   |D \xi| ^2    (|Du^\ez|^2+\kappa)^{   \alpha}   \, dx+ \eta \int_{U}
    ( f^\ez)^2  (|Du^\ez|^2+\kappa)^{     \alpha-3}  \xi^2  \, dx.
\end{align*}
Applying $\Delta_\fz u^\ez =-\ez\Delta u^\ez-f^\ez$, by $2< \alpha <3$ we further obtain
\begin{align*}
K_5+K_6&=-2 \alpha(\alpha-2)\ez  \int_{U}
   \Delta  u^\ez  f^\ez  (|Du^\ez|^2+\kappa)^{    \alpha-3}  \xi^2  \, dx\\
   &\quad
+2\kappa ({ \alpha-2}) ( \alpha-3)\ez \int_{U}\Delta  u^\ez   f^\ez  (|Du^\ez|^2+\kappa)^{    \alpha-4 }  \xi^2  \, dx\\
& \quad -2 \alpha(\alpha-2)   \int_{U}
   ( f^\ez)^2  (|Du^\ez|^2+\kappa)^{    \alpha-3}  \xi^2  \, dx\\
   &\quad
+2\kappa ({ \alpha-2}) ( \alpha-3) \int_{U}   ( f^\ez)^2  (|Du^\ez|^2+\kappa)^{    \alpha-4 }  \xi^2  \, dx\\
&\le  C \ez  \left[\int_{U}
  ( \Delta  u^\ez )^2\xi^2\,dx\right]^{1/2 }\left[\int_U( f^\ez)^2  (|Du^\ez|^2+\kappa)^{   2 \alpha-6}  \xi^2  \, dx\right]^{1/2} \\
& \quad -2 \alpha(\alpha-2)   \int_{U}
   ( f^\ez)^2  (|Du^\ez|^2+\kappa)^{    \alpha-3}  \xi^2  \, dx.
\end{align*}
Combining all estimates together we arrive at the desired result.

\medskip
{\it Proof of (iv): Case $3/2<  \alpha <2$.}
By Lemma \ref{identity3} with $\kappa>0$ and $3/2<  \alpha<2$, it suffices to estimate
the term $K$ as in the (iii).
Write $K=K_1+\cdots+K_6$ as in (iii).
The estimates for $K_2,\cdots,K_4$ are the same as there.
For $K_5$ and $K_6$, applying $f^\ez =-\Delta_\fz u^\ez-\ez\Delta u^\ez$ we have
\begin{align*}
K_5+K_6
&=- 2\alpha(\alpha-2) \int_{U}
   (\Delta_\fz  u^\ez)^2    (|Du^\ez|^2+\kappa)^{    \alpha-3}  \xi^2  \, dx\\
   &\quad - 2\alpha(\alpha-2) \ez \int_{U}
    \Delta_\fz u^\ez \Delta u^\ez    (|Du^\ez|^2+\kappa)^{    \alpha-3}  \xi^2  \, dx\\
   &\quad +2\kappa ({ \alpha-2}) ( \alpha-3)\int_{U} (\Delta_\fz u^\ez)^2     (|Du^\ez|^2+\kappa)^{    \alpha-4 }  \xi^2  \, dx\\
   &\quad +2\kappa ({ \alpha-2}) ( \alpha-3)\ez\int_{U}  \Delta_\fz u^\ez  \Delta u^\ez    (|Du^\ez|^2+\kappa)^{    \alpha-4 }  \xi^2  \, dx\\
   &=K_{5,1}+K_{5,2}+K_{6,1}+K_{6,2}
\end{align*}
Note that
\begin{align*}
   K_{5,2}+K_{6,2}
   &\le C\ez \kappa^{ \alpha-5/2}\left[\int_{U}
      (\Delta u^\ez)^2\,dx\right]^{1/2}  \left[\int_{U}
    |D^2u^\ez Du^\ez|^2 \,dx\right]^{1/2}
\end{align*}

Moreover by
\begin{align*}
&-K_{5,1}-K_{6,1}+\int_{U}  |D^2u^\ez Du^\epsilon |^2  (|Du^\ez|^2+\kappa)^{    \alpha-3} [2(|Du^\ez|^2+\kappa)+({ \alpha-2}) |Du^\ez|^2]\xi^2 \, dx \\
&\quad\ge \int_{U}  |D^2u^\ez Du^\epsilon |^2  (|Du^\ez|^2+\kappa)^{    \alpha-3} [2(|Du^\ez|^2+\kappa)+({ \alpha-2}) |Du^\ez|^2]\xi^2 \, dx \\
&\quad\quad-  2\alpha(2-\alpha) \int_{U} |D^2u^\ez D u^\ez|^2 |Du^\ez|^2
    (|Du^\ez|^2+\kappa)^{    \alpha-3}  \xi^2  \, dx\\
   &\quad\quad -2\kappa   \alpha   (1-{ \alpha })\int_{U} |D^2u^\ez D u^\ez|^2 |Du^\ez|^2      (|Du^\ez|^2+\kappa)^{    \alpha-4 }  \xi^2  \, dx\\
   &\quad\ge \alpha(2\alpha-3 )\int_{U} |D^2u^\ez D u^\ez|^2 |Du^\ez|^2
    (|Du^\ez|^2+\kappa)^{    \alpha-3}  \xi^2  \, dx\\
    &\quad\quad  +2\kappa[ 1- \alpha (1-  \alpha ) ]\int_{U} |D^2u^\ez D u^\ez|^2       (|Du^\ez|^2+\kappa)^{    \alpha-3}  \xi^2  \, dx\\
    &\quad\ge \min\{\alpha(2\alpha-3 ),\frac12\} \int_{U} |D^2u^\ez D u^\ez|^2
    (|Du^\ez|^2+\kappa)^{   { \alpha-2}}  \xi^2  \, dx.
\end{align*}

Notice that $3/2<{ \alpha }<2$ implies that
$$\alpha(2\alpha-3 )> ( 2\alpha-3) \quad {\rm and}\quad 2 [ 1-   \alpha  (2-{ \alpha }) ]\ge \frac14 .$$
Letting $\eta <\frac1{16}(2\alpha-3)<\frac1{16} $, combining all estimates together we obtain the desired result.
\end{proof}

\begin{proof}[Proof of Lemma \ref{Duez}.]
Letting ${ \alpha }=2$ and $\kappa=0$ and replacing $\xi$ by $u^\ez \xi^3$ in Lemma \ref{identity2}
we have

\begin{align*} &
2 \int_{U}  |D^2u^\ez Du^\epsilon |^2   (u^\ez \xi^3)^2 \, dx
+ 2\ez\int_{U}  | \Delta u^\epsilon |^2  (u^\ez \xi^3)^2 \, dx
+\int_{U}\langle Du^{\ez}, D(u^\ez \xi^3)\rangle ^2|Du^\ez|^2  \,dx\\
&= -2\int_{U}      u^{\ez}_i (u^\ez\xi^3)_i  \Delta_\fz u^{\ez}  u^\ez \xi^3 \,dx   -  \int_{U}(u^\ez\xi^3)_{ik} u^{\ez}_k u^{\ez}_i |Du^\ez|^2 u^\ez  \xi^3  \,dx\nonumber\\
&\quad-2\int_{U} |Du^{\ez}| ^2   u^{\ez}_{ik}u^{\ez}_k (u^\ez \xi^3)_i u^\ez \xi^3 \,dx  -\int_{U}
 f^\ez \Delta u^\ez    (u^\ez \xi^3)^2 \, dx.
\end{align*}
With a slight calculation, it can be  further writen  as
\begin{align*} &2\int_{U}  |D^2u^\ez Du^\epsilon |^2   (u^\ez)^{2 }\xi^{6} \, dx
+ 2\ez\int_{U}  | \Delta u^\epsilon |^2    (u^\ez)^{2 }\xi^{6}  \, dx\\
&\quad + \int_{U}   |Du^\ez|^{   6}  \xi^6\,dx+15\int_{U} \langle Du^{\ez}, D \xi \rangle^2 |Du^\ez|^{2}   (u^\ez)^{2 }  \xi^4 \,dx\\
&= -  5\int_{U}   |Du^\ez| ^{  2  }
  \Delta_\fz u^{\ez}    u^\ez \xi^6 \,dx -6\int_{U}      u^{\ez}_i\xi_i \Delta_\fz u^{\ez}   (u^\ez)^{2 }\xi^{5}  \,dx\\
&\quad-12
 \int_{U} u^\ez_i\xi_i |Du^\ez| ^{ 4} u^\ez\xi^{5} \,dx  -  3 \int_{U}\xi_{ik} u^{\ez}_k u^{\ez}_i |Du^\ez|^ {  2} (u^\ez)^{2 }\xi^5  \,dx\nonumber\\
 &\quad-6 \int_{U}  |Du^\ez| ^{ 2    }  u^{\ez}_{ik}u^{\ez}_k \xi _i  (u^\ez)^{2 }\xi^5 \,dx -2\int_{U}
 f^\ez \Delta u^\ez     (u^\ez)^{2 }\xi^6  \, dx
\end{align*}

For the first two terms, we have
 \begin{align*}
&   - 5 \int_{U}   \Delta_\fz u^{\ez}  |Du^\ez|^{2 } u^\ez\xi^6 \,dx
   -6\int_{U}       u^{\ez}_i  \xi_i \Delta_\fz u^{\ez} (u^\ez )^2\xi^5 \,dx\\
   &\quad=    - 5\ez  \int_{U}   \Delta  u^{\ez}  |Du^\ez|^{2  } u^\ez\xi^6 \,dx
   -6\ez \int_{U}      u^{\ez}_i   \xi _i \Delta u^{\ez} (u^\ez )^2\xi^5\,dx\\
   &\quad \quad+ 5 \int_{U}   f^{\ez}  |Du^\ez|^{2 } u^\ez\xi^6 \,dx
   +6  \int_{U}      u^{\ez}_i   \xi _i  f^{\ez} (u^\ez )^2\xi^5\,dx\\
&\quad\le      \frac1{16} \ez \int_{U}( \Delta  u^{\ez})^2  (u^\ez)^2  \xi^6  \,dx+
C   \ez \int_{U} |Du^\ez|^4  \xi^6\,dx   +C   \ez \int_{U} |Du^\ez|^{2  }( u^\ez)^2 \xi^4|D\xi|^2 \,dx\\
&\quad\quad  + \frac1{16} \int_{U} |Du^\ez|^6  \xi^6 \,dx +
C
\int_{U} (f u^\ez )^{3/2}\xi^6 \,dx+
C
\int_{U} [f (u^\ez)^2 |D\xi|]^{6/5}\xi^6 \,dx\\
&\quad\le   \frac18\ez \int_{U}( \Delta  u^{\ez})^2   (u^\ez)^2  \xi^6   \,dx+  \frac18 \int_{U} |Du^\ez|^6   \xi^6\,dx\\
&\quad\quad +
C
\int_{U} \xi^6 (f u^\ez )^{3} \,dx+
C
\int_{U}    (u^\ez) ^{6 } |D\xi| ^{6 } \,dx\\
&\quad\quad+C   \ez^{3} \int_{U}   \xi^6 \,dx +C   \ez^{3} \int_{U} |D\xi|^3 ( u^\ez)^3 |\xi|^3 \,dx.
    \end{align*}

By Young's inequality, we have
 \begin{align*}
&-12
 \int_{U} u^\ez_i\xi_i |Du^\ez| ^4 u^\ez\xi^5 \,dx - 3\int_{U}\xi_{ik} u^{\ez}_k u^{\ez}_i |Du^\ez|^ {  2  } (u^\ez)^{2 }\xi^5  \,dx\nonumber\\
 &\quad\le     \frac18 \int_{U} |Du^\ez|^6 \xi^6\,dx+ C
\int_{U}[| D\xi|^2+|D^2\xi||\xi|]^3(u^\ez )^6 \,dx.
 \end{align*}
Similarly,
 \begin{align*}
  & -6 \int_{U} |Du^{\ez}| ^{2 } u^{\ez}_{ik}u^{\ez}_k \xi _i (u^\ez)^2\xi^5 \,dx\\
&\quad\le \frac18 \int_{U}   |D^ 2u^{\ez} D u^{\ez}|^2 (u^\ez)^2\xi^6 \,dx+
 C \int_{U} |Du^{\ez}| ^4  (u^\ez)^2 |D\xi|^2\xi^4 \,dx\\
&\quad\le \frac18 \int_{U}   |D^ 2u^{\ez} D u^{\ez}|^2 (u^\ez)^2\xi^6 \,dx+
\frac18 \int_{U} |Du^{\ez}| ^6  \xi^6 \,dx+
 C  \int_{U}    (u^\ez)^6 |D\xi|^6  \,dx.
      \end{align*}

Finally,
by integration by parts we obtain
\begin{align*}
   -&\int_{U}
 f^\ez \Delta u^\ez    (u^\ez\xi^3)^2 \, dx\\
 &=   \int_{U}
    u^\ez _i   [f^\ez (u^\ez\xi^3)^2]_i\, dx\\
    &=    \int_{U}
    u^\ez _i   f^\ez_i (u^\ez\xi^3)^2  \, dx
        +2\int_{U}
    |D u^\ez |^2   f^\ez     u^\ez \xi ^{6}  \, dx +  6\int_{U}
  u^\ez _i      f^\ez    (u^\ez)^2 \xi_i \xi^{5}  \, dx\\
  &\le \int_{U}
    |D u^\ez|    |Df^\ez|  (u^\ez\xi^3)^2  \, dx+\frac18   \int_U |D u^\ez|^6\xi^6\,dx\\
&\quad+C \int_U (f^\ez u^\ez )^{3/2}  \xi^6\,dx+  C  \int_U  (f^\ez)^{6/5}    (u^\ez)^{12/5}|D\xi|^{6/5}|\xi|^{24/5}\,dx\\
  &\le\frac18  \int_U |D u^\ez|^6\xi^6\,dx+ \int_{U}
    |D u^\ez|    |Df^\ez|  (u^\ez\xi^3)^2  \, dx\\
&\quad+C \int_U (f^\ez u^\ez )^{3/2}   \xi^6\,dx+ C \int_U   (u^\ez)^{6} |D\xi|^6\,dx.
\end{align*}
Combining all estimates together with we obtain  the desired result.
\end{proof}

\begin{proof}[Proof of Lemma \ref{flat}.] Without loss of generality, we may assume that
$P(x)= cx_2$. Then  $ |c|= |DP|$, $DP=c{\bf e}_2$ and $\langle Du^\epsilon  , DP\rangle= cu^\epsilon _2$.
 Replacing  $\xi$ by $(u-cx_2)\xi$ in Lemma \ref{identity2}, letting ${ \alpha }=4$ and $\kappa=0$ we have
\begin{align*} &4\int_{U}  |D^2u^\ez Du^\epsilon |^2 |Du^\ez|^4  (u-cx_2)^2\xi^2 \, dx
+ 4\ez\int_{U}  |\Delta u^\ez  |^2    |Du^\ez|^4  (u-cx_2)^2\xi^2 \, dx\\
&\quad+\int_{U}\langle Du^{\ez}, D[(u-cx_2)\xi]\rangle ^2|Du^\ez|^6  \,dx\\
&= -4\int_{U}    |Du^{\ez}|  ^4      u^{\ez}_i [(u-cx_2)\xi]_i  \Delta_\fz u^{\ez} (u-cx_2)\xi \,dx\\
&\quad   -  \int_{U}[(u-cx_2)\xi]_{ik} u^{\ez}_k u^{\ez}_i |Du^\ez|^6 (u-cx_2)\xi \,dx\nonumber\\
&\quad-2\int_{U} |Du^{\ez}| ^6 u^{\ez}_{ik}u^{\ez}_k[(u-cx_2)\xi]_i (u-cx_2)\xi \,dx\nonumber \\
&\quad -\int_{U}
 f^\ez \Delta u^\ez   |Du^\ez|^4  [(u-cx_2)\xi]^2 \, dx\\
&=J_1+\cdots +J_4.
\end{align*}

Firstly, we note that
 \begin{align*}
&\int_{U}\langle Du^{\ez}, D[(u-cx_2)\xi]\rangle ^2|Du^\ez|^6  \,dx\\
&\quad=\int_{U}[\langle Du^{\ez},  D  u-DP  \rangle ^2]\xi ^2+
2 \langle Du^{\ez},  D  u-DP  \rangle \langle Du^{\ez},  D  \xi  \rangle\\
&\quad\quad\quad+ \langle Du^{\ez},  D \xi  \rangle ^2(u-cx_2)^2] |Du^\ez|^6 \,dx \\
&\quad\ge \frac12\int_{U}\langle Du^{\ez},  D  u-DP  \rangle ^2|Du^\ez|^6 \xi ^2 \,dx-
\int_{U}\langle Du^{\ez},  D\xi  \rangle ^2|Du^\ez|^6 (u^\ez-cx_2)^2 \,dx\\
&\quad\ge \frac12\int_{U}\langle Du^{\ez},  D  u-DP  \rangle ^2|Du^\ez|^6 \xi ^2 \,dx-
\int_{U}| D\xi|^2|Du^\ez|^8 (u^\ez-cx_2)^2 \,dx.
\end{align*}

 Next, we have
 \begin{align*}
J_1+J_3
&= -4\int_{U}    |Du^{\ez}|  ^4       u^{\ez}_i  (u-cx_2)_i \Delta_\fz u^{\ez} (u^\ez-cx_2)\xi^2 \,dx\\
 &\quad-4\int_{U}    |Du^{\ez}|  ^4       u^{\ez}_i  \xi_i \Delta_\fz u^{\ez} (u^\ez-cx_2)^2\xi \,dx\\
 &\quad
 -2\int_{U} |Du^{\ez}| ^6 u^{\ez}_{ik}u^{\ez}_k (u-cx_2)_i  (u-cx_2)\xi^2 \,dx\\
 &\quad-2\int_{U} |Du^{\ez}| ^6 u^{\ez}_{ik}u^{\ez}_k \xi _i (u-cx_2)^2\xi \,dx\\
& \le\left[\int_{U}   |D^2u^\ez Du^\ez|^2     \xi^2 \,dx\right]^{1/2}
\left[\int_{U}    |Du^{\ez}|  ^{12}       |Du-DP|^2    (u^\ez-cx_2)^2\xi^2 \,dx\right]^{1/2}\\
&\quad  + \int_{U}   |D^2u^\ez Du^\ez|^2 |Du^{\ez}|  ^4  (u^\ez-cx_2)^2  \xi^2 \,dx+
C\int_{U}    |Du^{\ez}|  ^{8}        |D\xi|^2    (u^\ez-cx_2)^2  \,dx
\end{align*}

We also have
 \begin{align*}
 J_2
 &=   -  \int_{U}  u^\ez  _{ik} u^{\ez}_k u^{\ez}_i |Du^\ez|^6 (u-cx_2)\xi^2 \,dx -2\int_{U}  (u^\ez-cx_2)  _{i }\xi_k u^{\ez}_k u^{\ez}_i |Du^\ez|^6 (u-cx_2)\xi \,dx\\
 &\quad -  \int_{U}  \xi  _{ik} u^{\ez}_k u^{\ez}_i |Du^\ez|^6 (u-cx_2)^2\xi  \,dx\nonumber \\
 &\le \left[\int_{U}   |D^2u^\ez Du^\ez|^2    \xi^2 \,dx \right]^{1/2}
\left[ \int_{U}    |Du^{\ez}|  ^{12}       |Du-DP|^2    (u^\ez-cx_2)^2\xi^2 \,dx\right]^{1/2}\\
 &\quad + \frac18\int_{U}   \langle Du^\ez, D u^\ez-DP\rangle^2 |Du^{\ez}|  ^6   \xi^2 \,dx\\
 &\quad
  +
 \int_{U}    |Du^{\ez}|  ^8        (|D\xi|^2 +|D^2\xi||\xi|)  (u^\ez-cx_2)^2  \,dx.
 \end{align*}
By integration by parts, we obtain
\begin{align*}
J_4
 &=   \int_{U}
    u^\ez _i  [f^\ez|Du^\ez|^4   (u-cx_2)^2\xi ^2 ]_i\, dx\\
    &=    2\int_{U}
    \Delta_\fz u^\ez    f^\ez|Du^\ez|^2   (u-cx_2)^2\xi ^2  \, dx+\int_{U}
    u^\ez _i   f^\ez_i|Du^\ez|^4   (u-cx_2)^2\xi ^2  \, dx  \\
      &\quad  +2\int_{U}
    u^\ez _i   f^\ez |Du^\ez|^4 (u-cx_2)_i  (u-cx_2) \xi ^2  \, dx +  2\int_{U}
  u^\ez _i      f^\ez|Du^\ez|^4  (u-cx_2)^2\xi _i\xi  \, dx\\
  &\le  \int_{U}   |D^2u^\ez Du^\ez|^2 |Du^{\ez}|  ^4  (u^\ez-cx_2)^2  \xi^2 \,dx+ \frac18\int_{U}   \langle Du^\ez, D u^\ez-DP\rangle^2 |Du^{\ez}|  ^6   \xi^2 \,dx\\
  &\quad +
C\int_{U}      |Du^{\ez}|  ^2 (f^\ez)^2         \xi^2    (u^\ez-cx_2)^2  \,dx +C\int_{U}         |Du^{\ez}|  ^8     |D\xi|^2     (u^\ez-cx_2)^2  \,dx\\
&\quad
+ C\int_{U}         |Du^{\ez}|  ^4     f^\ez_i u^\ez_i \xi^2     (u^\ez-cx_2)^2  \,dx.
\end{align*}

Combining all estimates together we arrive at
\begin{align*}
& \int_{U}\langle Du^{\ez}, D[(u-cx_2)\xi]\rangle ^2|Du^\ez|^6  \,dx\\
& \quad\le \left[ \int_{U}   |D^2u^\ez Du^\ez|^2   \xi^2 \,dx \right]^{1/2}
\left[ \int_{U}    |Du^{\ez}|  ^{12}       |Du-DP|^2    (u^\ez-cx_2)^2\xi^2 \,dx\right]^{1/2}\\
&\quad\quad+C  \int_{U}   [|Du^{\ez}|  ^8(|D\xi|^2 +|D^2\xi||\xi|) +|Du^{\ez}|  ^2 (f^\ez)^2\xi^2]   (u^\ez-cx_2)^2  \,dx\\
&\quad\quad+C \left| \int_{U}  |Du^{\ez}|  ^4  f^\ez_i u^\ez_i\xi   \    (u^\ez-cx_2)^2  \,dx\right|
 \end{align*}
 as desired.
\end{proof}

\bigskip
 \noindent {\bf Acknowledgment.} H. Koch and Y. Zhang have been
 partially supported by the Hausdorff Center for Mathematics.  Y. Zhou
 would like to thank the supports of von Humboldt Foundation, and
 National Natural Science of Foundation of China (No. 11522102).

\end{document}